\documentclass[11pt,reqno]{amsart}
\usepackage[margin=1.10in]{geometry}
\usepackage{amsmath,amssymb,amsthm,mathtools}
\numberwithin{equation}{section}
\usepackage{microtype}
\usepackage{xcolor}
\usepackage[hidelinks]{hyperref}

\newtheorem{theorem}{Theorem}[section]
\newtheorem{proposition}[theorem]{Proposition}
\newtheorem{lemma}[theorem]{Lemma}
\newtheorem{corollary}[theorem]{Corollary}
\theoremstyle{remark}
\newtheorem{remark}[theorem]{Remark}

\newcommand{\Irr}{\operatorname{Irr}}
\newcommand{\cod}{\operatorname{cod}}
\newcommand{\GL}{\operatorname{GL}}

\newcommand{\PSL}{\operatorname{PSL}}
\newcommand{\GU}{\operatorname{GU}}

\newcommand{\SO}{\operatorname{SO}}
\newcommand{\PSO}{\operatorname{PSO}}
\newcommand{\Spin}{\operatorname{Spin}}
\newcommand{\OmegaG}{\operatorname{\Omega}}
\newcommand{\Sp}{\operatorname{Sp}}
\newcommand{\PSp}{\operatorname{PSp}}
\newcommand{\POmega}{\operatorname{P\Omega}}
\newcommand{\PSU}{\operatorname{PSU}}
\newcommand{\SU}{\operatorname{SU}}
\newcommand{\F}{\mathbb F}
\newcommand{\C}{\mathbb C}

\newcommand{\Ker}{\operatorname{Ker}}

\title[The codegree isomorphism conjecture]{The Codegree Isomorphism Conjecture for Finite Simple Groups}
\author{YONG YANG}
\address{Department of Mathematics, Texas State University, San Marcos, TX 78666, USA.}
\email{yang@txstate.edu}
\date{}


\begin{document}
\begin{abstract}
We prove the codegree isomorphism conjecture for finite simple groups. Thus every nonabelian finite simple group is determined, among finite groups, by its set of irreducible character codegrees.
\end{abstract}
\maketitle

\section{Introduction}
For a finite group $G$ and $\chi\in\Irr(G)$, we define
\[
  \cod(\chi)=\frac{|G:\Ker(\chi)|}{\chi(1)},
\qquad
  \cod(G)=\{\cod(\chi):\chi\in\Irr(G)\}.
\]
We also write
\[
  \operatorname{cd}(G)=\{\chi(1):\chi\in\Irr(G)\}
\]
for the set of irreducible character degrees of $G$.  If $m$ is a positive integer and $p$ is a prime, then $m_p$ and $m_{p'}$ denote the $p$-part and the $p'$-part of $m$, respectively.  We write $M(G)$ for the Schur multiplier of $G$.

The codegree isomorphism conjecture was proposed by Qian and appears as Problem~20.79 in the Kourovka Notebook \cite{Kourovka}. It asks whether
\[
   \cod(G)=\cod(H)
\]
forces $G\cong H$ whenever $H$ is a nonabelian finite simple group and $G$ is a finite group.  We prove the conjecture in full.

\begin{theorem}\label{thm:global}
Let $H$ be a nonabelian finite simple group and let $G$ be a finite group.  If
\[
  \cod(G)=\cod(H),
\]
then $G\cong H$.
\end{theorem}

The proof of Theorem~\ref{thm:global} is completed in Section~\ref{sec:completion}, after the simple-group families have all been treated.

This theorem completes a sequence of earlier results and reductions.  Dolorfino, Martin, Slonim, Sun and Yang first established the sporadic case \cite{DMSSYsporadic} and then the alternating case \cite{DMSSYalternating}.  Hung and Moret\'o proved the order-containment and simple-quotient results \cite[Theorem~C and Theorem~8.3]{HMI}; in the sequel \cite[Theorem~3.1]{HMII}, they reduced the remaining containment problem to the conditions used below.  Tong-Viet proved the cases in \cite[Theorem~1.1]{TV} and reduced the remaining projective symplectic groups to a natural-module extension \cite[Theorem~5.1]{TV}.  The Tits group ${}^2F_4(2)'$ is covered by Wang, Zhang, Zhang and Chen \cite{WZZC}.  It therefore remains to consider the following classical groups.

For clarity, the table below shows how the simple-group families are covered.  Results proved in this paper are identified by theorem or corollary number, and their proofs are located in the sections indicated immediately after their statements below; the other entries refer to earlier work:
\[
\begin{array}{c|l}
\text{simple group family} & \text{}\\ \hline
A_n,\ n\ge5 & \text{\cite{DMSSYalternating}}\\
\text{sporadic groups} & \text{\cite{DMSSYsporadic}}\\
\PSL_2(q) & \text{\cite{BAK}}\\
\PSL_3(q) & \text{\cite{LY}}\\
\text{exceptional Lie type covered by \cite[Theorem~1.1]{TV} and }\PSL_n(q) & \text{\cite[Theorem~1.1]{TV}}\\
{}^2F_4(2)' & \text{\cite{WZZC}}\\
\PSU_3(q) & \text{\cite{LY}}\\
\PSU_n(q),\ n\ge4 & \text{Corollary~\ref{cor:Umain}}\\
\PSp_4(q),\ \PSp_6(q) & \text{\cite[Theorem~5.1]{TV}}\\
\PSp_{2n}(q),\ n\ge4 & \text{Theorem~\ref{thm:sympall}}\\
\OmegaG_{2n+1}(q),\ q\text{ odd},\ n\ge3 & \text{Corollary~\ref{cor:Bmain}}\\
\POmega_{2n}^{\pm}(q),\ q\text{ even},\ n\ge4 & \text{Corollary~\ref{cor:mainA}}\\
\POmega_{2n}^{\pm}(q),\ q\text{ odd} & \text{Theorem~\ref{thm:Doddallcodegree}}
\end{array}
\]

For the unitary groups, odd-dimensional orthogonal groups, and even-dimensional orthogonal groups in even characteristic, we establish the character-theoretic condition required in \cite{HMII} and obtain the corresponding stronger containment theorem.  For even-dimensional orthogonal groups in odd characteristic we prove the two remaining cases directly.  The prime-central case is the delicate one: we use semisimple characters of spin groups, with separate arguments in ranks four, five, and six and with the low-degree classification of H.~N.~Nguyen \cite[Theorem~1.4]{Ng} in the exceptional small-rank cases.  Together with the large-rank arguments this completes both signs of type $D_n$ for every odd $q$.  For symplectic groups, the reduction in \cite[Theorem~5.1]{TV} is finished by projective induction in even characteristic (including the isolated nonsplit extension over $\Sp_8(2)$) and by a kernel obstruction in odd characteristic.

We will use the following two statements from the reduction in \cite{HMII}.  We say that Conjecture A holds for a nonabelian finite simple group $H$ if, for every nontrivial finite group $G$,
\[
  \cod(G)\subseteq\cod(H) \quad\Longrightarrow\quad G\cong H.
\]

\medskip
\noindent\textbf{Conjecture B.}
Let $H$ be a nonabelian finite simple group. Suppose that $G$ is perfect, $N\lhd G$ is an elementary abelian $p$-group, $G/N\cong H$, and $N$ is a faithful irreducible $H$-module. Then there exists a faithful $\chi\in\Irr(G)$ such that
\[
  \chi(1)_p<|N|.
\]
For all families in which we use this condition, it serves only as an intermediate step toward the codegree theorem.

\begin{theorem}\label{thm:mainB}
Let $q=2^e$, let $n\ge4$, and let
\[
 H=\POmega_{2n}^{\varepsilon}(q),\qquad \varepsilon\in\{+,-\},
\]
be simple. Then Conjecture B holds for $H$.
\end{theorem}
The proof of Theorem~\ref{thm:mainB} is given in Section~\ref{sec:evenorthproof}.

\begin{corollary}\label{cor:mainA}
Let $H=\POmega_{2n}^{\varepsilon}(q)$ be simple, where $q$ is even and $n\ge4$. If $G$ is a nontrivial finite group and $\cod(G)\subseteq\cod(H)$, then $G\cong H$.
\end{corollary}
The proof of Corollary~\ref{cor:mainA} is given in Section~\ref{sec:evenorthproof}.

\begin{theorem}\label{thm:Bmain}
Let $q$ be odd and let $H=\OmegaG_{2n+1}(q)$ be simple, with $n\ge3$. Then Conjecture B holds for $H$.
\end{theorem}
The proof of Theorem~\ref{thm:Bmain} is given in Section~\ref{sec:B}.

\begin{corollary}\label{cor:Bmain}
Let $q$ be odd, let $H=\OmegaG_{2n+1}(q)$ be simple with $n\ge3$, and let $G$ be a nontrivial finite group. If $\cod(G)\subseteq\cod(H)$, then $G\cong H$.
\end{corollary}
The proof of Corollary~\ref{cor:Bmain} is given in Section~\ref{sec:B}.

\begin{theorem}\label{thm:Umain}
Let $n\ge3$ and let $H=\PSU_n(q)$ be nonabelian simple. Then Conjecture B holds for $H$.
\end{theorem}
The proof of Theorem~\ref{thm:Umain} is given in Section~\ref{sec:unitary}.

\begin{corollary}\label{cor:Umain}
Let $n\ge3$, let $H=\PSU_n(q)$ be nonabelian simple, and let $G$ be a nontrivial finite group. If $\cod(G)\subseteq\cod(H)$, then $G\cong H$.
\end{corollary}
The proof of Corollary~\ref{cor:Umain} is given in Section~\ref{sec:unitary}.

\section{Reductions for nonabelian finite simple groups}
We recall two consequences that will be used without further comment.

\begin{lemma}\label{lem:HM}
Let $H$ be a finite nonabelian simple group and let $G,N,p$ satisfy the hypotheses of Conjecture B.
\begin{enumerate}
\item If
\[
 |N|>\sqrt{|H|_p},
\]
then there is a faithful $\chi\in\Irr(G)$ with $\chi(1)_p<|N|$.
\item If $H$ is a simple group of Lie type of defining characteristic $\ell\ne p$, then the conclusion of Conjecture B holds.
\item If the extension $1\to N\to G\to H\to1$ splits, then $G$ has a faithful irreducible character of $p'$-degree.
\end{enumerate}
\end{lemma}

\begin{proof}
These are \cite[Propositions 4.2, 4.3 and 5.1]{HMII}, respectively.
\end{proof}

We shall also use the following elementary finite-field descent observation.

\begin{lemma}\label{lem:descent}
Let $p$ be a prime, let $X$ be a finite group, and let $V$ be an irreducible $\F_pX$-module.  Let $U$ be an absolutely irreducible constituent of
$V\otimes_{\F_p}\overline{\F}_p$.  If the smallest field of definition of $U$ is $\F_{p^a}$, then
\[
 |V|=p^{a\dim U}.
\]
\end{lemma}

\begin{proof}
Put $E=\operatorname{End}_{\F_pX}(V)$.  By Schur's lemma and Wedderburn's theorem, $E=\F_{p^e}$ for some $e$.  With the commuting $E$-structure, $V$ is an absolutely irreducible $EX$-module.  Hence $V\otimes_{\F_p}\overline{\F}_p$ is the direct sum of the $e$ distinct Frobenius conjugates of one absolutely irreducible constituent.  Their Frobenius orbit has length $e$, so their smallest field of definition is $\F_{p^e}$.  Since $U$ has smallest field of definition $\F_{p^a}$, we have $e=a$.  Therefore
\[
 \dim_{\F_p}V=e\dim U=a\dim U,
\]
and the assertion follows.
\end{proof}

\section{Small modules for even-dimensional orthogonal groups in characteristic $2$}\label{sec:smallmodule}
Throughout this section let
\[
 q=2^e,\qquad H=\POmega_{2n}^{\varepsilon}(q),\qquad n\geq5.
\]
For even $q$ the natural projective module has dimension $2n$.  Moreover, in this range the simple group $H$ has trivial Schur multiplier: the multiplier formula of Kleidman and Liebeck \cite[Theorem~5.1.4 and Tables~5.1.A, 5.1.D]{KL} gives generic multiplier order
$\gcd(4,q^n-\varepsilon)=1$, and its exceptional multiplier table contains no group of type $D_n$ or ${}^2D_n$ with $n\geq5$ and $q$ even.  Thus the finite group of simply connected type to which \cite[Lemma 4.2]{GT} applies may be identified with $H$.

The $2$-part of the orthogonal group order is
\[
 |H|_2=q^{n(n-1)}.
\]
Put
\[
 t=\frac{n(n-1)}2.
\]

\begin{proposition}\label{prop:natural}
Let $N$ be a faithful irreducible $\F_2H$-module satisfying
\[
 |N|\leq\sqrt{|H|_2}=q^t.
\]
Then, after twisting the action by an automorphism of $H$, $N$ is the natural module viewed as an $\F_2$-module.  In particular,
\[
 |N|=q^{2n}.
\]
\end{proposition}

\begin{proof}
Let $D$ be the smallest dimension of a nontrivial absolutely irreducible representation of the simply connected group of type $D_n$ in characteristic $2$.  Here
\[
 D=2n.
\]
Indeed the natural module has dimension $2n$, and the small-module list \cite[Proposition~5.4.11 and Table~5.4.A]{KL} shows that no smaller nontrivial projective module occurs in our range.

Apply Guralnick and Tiep \cite[Lemma~4.2]{GT} to the irreducible $\F_2H$-module $N$.  We examine alternatives (i)--(v) there.  In the terminology of that lemma, a ``restricted module'' includes a Frobenius twist of a restricted module.

First suppose alternative (v) holds.  Then
\[
 |N|\geq q^{2D^2}=q^{8n^2},
\]
which contradicts $|N|\le q^t$.

Alternative (iv) occurs only for ${}^3D_4(q)$ and is irrelevant here.

Suppose alternative (ii) holds.  Then an absolutely irreducible constituent has the form
\[
 U\cong W\otimes W^{(e/2)},
\]
where $W$ is a nontrivial restricted module, and the smallest field of definition of $U$ is $\F_{q^{1/2}}$.  Hence $\dim W\ge D=2n$ and
\[
 \dim U\ge D^2=4n^2.
\]
By Lemma~\ref{lem:descent},
\[
 |N|=(q^{1/2})^{\dim U}\ge q^{2n^2}>q^t,
\]
a contradiction.

Now assume alternative (i).  Then an absolutely irreducible constituent $U$ is a tensor product of at most two nontrivial restricted modules, and its smallest field of definition is $\F_q$.  If two nontrivial factors occur, then
\[
 \dim U\ge D^2=4n^2,
\]
so Lemma~\ref{lem:descent} gives $|N|\ge q^{4n^2}>q^t$, again impossible.  Thus $U$ itself is restricted and
\[
 |N|=q^{\dim U},\qquad \dim U\le t.
\]

We now apply \cite[Proposition~5.4.11]{KL}.  For an orthogonal group of natural dimension $d=2n$, that proposition applies to every nontrivial irreducible projective module of dimension less than $d^2/2-1=2n^2-1$, and our bound $t<n^2<2n^2-1$ lies well inside this range.  In type $D_n$ the possible modules in Table 5.4.A, apart from the natural module, are a section of $\Lambda^2$ of dimension
\[
 \frac{d(d-1)}2-\gcd(2,n)=n(2n-1)-\gcd(2,n),
\]
and, only when $n<7$ and $q$ is even, a half-spin module of dimension $2^{n-1}$.  The first number is greater than $t$ for every $n\ge5$, while
\[
 2^{4}=16>10=t\quad(n=5),\qquad
 2^{5}=32>15=t\quad(n=6).
\]
Consequently $U$ is quasiequivalent to the natural module, where quasiequivalent representations are those obtained from one another by an automorphism of the group together with an equivalence of the underlying modules.  Twisting the chosen identification $G/N\cong H$ by the corresponding automorphism of $H$, we may assume that $U$ is natural.  Therefore
\[
 |N|=q^{2n}.
\]

It remains to exclude alternative (iii), which can occur only for the twisted group ${}^2D_n(q)$ in the present situation.  Here $U$ is restricted and its smallest field of definition is $\F_{q^2}$.  By Lemma~\ref{lem:descent},
\[
 |N|=q^{2\dim U},
\]
so the hypothesis $|N|\le q^t$ gives
\[
 \dim U\le \frac t2.
\]
In particular $\dim U<t<2n^2-1$, and the same application of \cite[Proposition 5.4.11]{KL} shows that $U$ must be quasiequivalent to the natural module (the other entries of Table 5.4.A are already larger than $t$, hence a fortiori larger than $t/2$).  This is impossible in alternative (iii).  Indeed, in the twisted case the proof of \cite[Lemma~4.2]{GT} places alternative (iii) in Case~2, where the restricted module $W$ satisfies
\[
 W\not\cong W^{\tau_0},
\]
with $\tau_0$ the graph automorphism.  The natural $D_n$ highest weight is fixed by this graph automorphism, and the same is true of its Frobenius twists.  Thus a module quasiequivalent to the natural module cannot yield alternative (iii).

All alternatives except (i) with one natural factor have been eliminated, and the proposition follows.
\end{proof}

\begin{remark}\label{rem:GTexp}
The exponent in alternative (v) of \cite[Lemma 4.2]{GT} is $2D^2$, not $2D$.  This is also clear from the proof of that lemma, where the estimates are repeatedly compared with $q^{2D^2}$.
\end{remark}

\section{Natural-module extensions for even-dimensional orthogonal groups in characteristic $2$}\label{sec:naturalextension}
We next handle all extensions by the natural module.  The argument does not require the extension class itself to split.

\begin{lemma}\label{lem:stabilizer}
Let $V$ be the natural $2n$-dimensional orthogonal module for
$H=\POmega_{2n}^{\varepsilon}(q)$, where $q$ is even and $n\ge5$.  If $v\in V$ is nonsingular, then
\[
 H_v\cong \Sp_{2n-2}(q)
\]
and
\[
 [H:H_v]=q^{n-1}(q^n-\varepsilon).
\]
Moreover the Schur multiplier of $H_v$ is trivial.
\end{lemma}

\begin{proof}
The stabilizer result \cite[Proposition~4.1.7]{KL} gives the stabilizer of a nonsingular $1$-space in the even-characteristic orthogonal group as $\Sp_{2n-2}(q)$; the proof there identifies the induced symplectic space with $v^\perp/\langle v\rangle$, proves surjectivity onto the symplectic group, and proves that the kernel is trivial \cite[Proposition 4.1.7]{KL}.  Since $q$ is even, an isometry stabilizing the $1$-space $\langle v\rangle$ must fix $v$: if $v$ is sent to $av$, preservation of the nonzero value $Q(v)$ gives $a^2=1$, hence $a=1$ in $\F_q$.  Thus this is also the vector stabilizer.

The index follows from the standard orders
\[
 |\OmegaG_{2n}^{\varepsilon}(q)|
 =q^{n(n-1)}(q^n-\varepsilon)\prod_{i=1}^{n-1}(q^{2i}-1)
\]
and
\[
 |\Sp_{2n-2}(q)|
 =q^{(n-1)^2}\prod_{i=1}^{n-1}(q^{2i}-1).
\]
Finally, for $m=n-1\ge4$ and even $q$, the multiplier theorem and table \cite[Theorem~5.1.4 and Table~5.1.A]{KL} give
\[
 |M(P\Sp_{2m}(q))|=(2,q-1)=1,
\]
and no group in this range occurs in the exceptional multiplier table \cite[Theorem 5.1.4 and Tables 5.1.A, 5.1.D]{KL}.  Since $q$ is even, $\Sp_{2m}(q)=P\Sp_{2m}(q)$, completing the proof.
\end{proof}

\begin{lemma}\label{lem:extend}
Let
\[
 1\longrightarrow V\longrightarrow G\longrightarrow H\longrightarrow1
\]
be an extension, where $V$ is the natural module from Lemma~\ref{lem:stabilizer}.  Let $v\in V$ be nonsingular.  Then there is a nonprincipal $\lambda\in\Irr(V)$ such that
\[
 I_G(\lambda)/V\cong H_v,
\]
and $\lambda$ extends to $I_G(\lambda)$.
\end{lemma}

\begin{proof}
Let $B$ be the nondegenerate alternating bilinear form polarizing the quadratic form on $V$, and fix a nonprincipal additive character
$\nu:\F_q^+\to\C^\times$.  Define
\[
 \lambda_v(x)=\nu(B(v,x))\qquad(x\in V).
\]
The trace pairing is nondegenerate, so $v\mapsto\lambda_v$ is an injective $H$-equivariant map from $V$ to $\Irr(V)$; since both sets have size $|V|$, it is a bijection.  Hence
\[
 I_H(\lambda_v)=H_v.
\]
Put $T=I_G(\lambda_v)$ and $S=T/V\cong H_v$.

For completeness we give the standard extension argument.  Choose a section $s:S\to T$.  Since $\lambda_v$ is $T$-invariant, the function
\[
 \alpha(x,y)=\lambda_v\bigl(s(x)s(y)s(xy)^{-1}\bigr)
\]
is a $2$-cocycle in $Z^2(S,\C^\times)$; its cohomology class is precisely the obstruction to extending $\lambda_v$ to $T$.  By Lemma~\ref{lem:stabilizer}, the Schur multiplier of $S$ is trivial, and therefore
\[
 H^2(S,\C^\times)=1.
\]
Thus $\alpha$ is a coboundary.  Modifying the values on the chosen section by a suitable $1$-cochain gives a linear character $\widetilde\lambda_v$ of $T$ extending $\lambda_v$.
\end{proof}

\begin{proposition}\label{prop:nonsplit}
Let $q$ be even, $n\ge5$, and let
\[
 1\longrightarrow N\longrightarrow G\longrightarrow
 \POmega_{2n}^{\varepsilon}(q)\longrightarrow1
\]
be a nonsplit extension in which $N$ is the natural module.  If $G$ is perfect, then $G$ has a faithful $\chi\in\Irr(G)$ satisfying
\[
 \chi(1)_2<|N|.
\]
In fact one may choose
\[
 \chi(1)=q^{n-1}(q^n-\varepsilon).
\]
\end{proposition}

\begin{proof}
Choose a nonsingular $v$ and let $\lambda=\lambda_v$ and $T=I_G(\lambda)$ as in Lemma~\ref{lem:extend}.  Extend $\lambda$ to a linear character $\widetilde\lambda\in\Irr(T)$.  By Clifford correspondence,
\[
 \chi=\widetilde\lambda^{\,G}
\]
is irreducible.  Lemma~\ref{lem:stabilizer} gives
\[
 \chi(1)=[G:T]=[H:H_v]=q^{n-1}(q^n-\varepsilon).
\]
Because $q$ is even, $q^n-\varepsilon$ is odd, and hence
\[
 \chi(1)_2=q^{n-1}<q^{2n}=|N|.
\]

It remains to prove faithfulness.  Put $K=\Ker(\chi)$.  Since $\chi_N$ is a sum of the $H$-conjugates of the nonprincipal character $\lambda$, the subgroup $K\cap N$ is the intersection of their kernels.  It is $H$-invariant.  As $N$ is an irreducible $H$-module and $\lambda$ is nonprincipal, this intersection is proper, and therefore
\[
 K\cap N=1.
\]
Now $KN/N\lhd G/N=H$.  Simplicity of $H$ gives either $K\le N$, in which case $K=1$, or $KN=G$.  In the latter case $K\cap N=1$, so $K$ is a complement to $N$ in $G$, contrary to the assumption that the extension is nonsplit.  Hence $K=1$.
\end{proof}

\section{Even-dimensional orthogonal groups in characteristic $2$}\label{sec:evenorthproof}
\begin{proof}[Proof of Theorem~\ref{thm:mainB}]
For $n=4$ the result is \cite[Proposition 4.3]{HMII}, so assume $n\ge5$.  Let $G,N,p$ satisfy the hypotheses of Conjecture B for
\[
 H=\POmega_{2n}^{\varepsilon}(q),\qquad q=2^e.
\]
By Lemma~\ref{lem:HM}(2), we may assume that $p=2$.  By Lemma~\ref{lem:HM}(1), we may further assume
\[
 |N|\le\sqrt{|H|_2}=q^{n(n-1)/2}.
\]
Proposition~\ref{prop:natural} now shows that, after twisting the quotient identification by an automorphism of $H$, $N$ is the natural module and
\[
 |N|=q^{2n}.
\]
If the extension splits, Lemma~\ref{lem:HM}(3) gives a faithful irreducible character of odd degree.  If it does not split, Proposition~\ref{prop:nonsplit} gives a faithful $\chi\in\Irr(G)$ with
\[
 \chi(1)_2=q^{n-1}<q^{2n}=|N|.
\]
Thus Conjecture B holds.
\end{proof}

\begin{proof}[Proof of Corollary~\ref{cor:mainA}]
Theorem~\ref{thm:mainB} proves Conjecture B for $H$.  Since $H$ is not an even-dimensional orthogonal group in odd characteristic, Theorem~3.1 of \cite{HMII} implies Conjecture A for $H$.
\end{proof}

\section{Odd-dimensional orthogonal groups in odd characteristic}\label{sec:B}
Let
\[
 H=\OmegaG_{2n+1}(q),\qquad q=p^f\ \text{odd}.
\]
Conjecture B is already known for $\OmegaG_7(q)$ \cite[Proposition 4.3]{HMII}.  We therefore concentrate on $n\ge4$.

The defining-characteristic part of the group order is
\[
 |H|_p=q^{n^2}.
\]
Thus Lemma~\ref{lem:HM}(1) allows us to assume throughout that
\begin{equation}\label{eq:Bthreshold}
 |N|\le q^{n^2/2}.
\end{equation}

\begin{proposition}\label{prop:Bnatural}
Let $n\ge4$, and let $N$ be a faithful irreducible $\F_pH$-module satisfying \eqref{eq:Bthreshold}. Then, after twisting the quotient identification by an automorphism of $H$, $N$ is the natural $(2n+1)$-dimensional orthogonal module over $\F_q$, viewed as an $\F_p$-module. In particular,
\[
 |N|=q^{2n+1}.
\]
For $n=4$ no such $N$ satisfying \eqref{eq:Bthreshold} exists.
\end{proposition}

\begin{proof}
Inflate $N$ to the simply connected group of type $B_n$.  The smallest dimension of a nontrivial absolutely irreducible defining-characteristic representation is
\[
 D=2n+1.
\]
This follows from the natural module and the small-module list in \cite[Proposition 5.4.11 and Table 5.4.A]{KL}.  Lemma~4.2 of \cite{GT} therefore gives
\[
 |N|\ge q^D.
\]
When $n=4$, this yields $|N|\ge q^9>q^8=q^{n^2/2}$, proving the final assertion.

Assume $n\ge5$.  Since type $B_n$ is untwisted, only alternatives (i), (ii), and (v) of \cite[Lemma 4.2]{GT} can occur.  Alternative (v) gives
\[
 |N|\ge q^{2D^2}>q^{n^2/2},
\]
contrary to \eqref{eq:Bthreshold}.  In alternative (ii), an absolutely irreducible constituent has the form
\[
 U\cong W\otimes W^{(f/2)},
\]
with $\dim W\ge D$, and its smallest field of definition is $\F_{q^{1/2}}$.  Hence
\[
 |N|\ge q^{D^2/2}>q^{n^2/2},
\]
again a contradiction.

We are therefore in alternative (i).  If two nontrivial restricted tensor factors occur, then $\dim U\ge D^2$ and hence $|N|\ge q^{D^2}$, impossible.  Thus $U$ is a single restricted module, its smallest field of definition is $\F_q$, and
\[
 |N|=q^{\dim U},\qquad \dim U\le \frac{n^2}{2}.
\]
The small-module result \cite[Proposition~5.4.11 and Table~5.4.A]{KL} now applies.  For type $B_n$ with $q$ odd, every module in this range is quasiequivalent to the natural module or to one of the listed small exceptions.  The module $\Lambda^2V$ listed there has dimension
\[
 \frac12(2n+1)(2n)=n(2n+1)>\frac{n^2}{2}.
\]
The spin module occurs in the table only for $n\le6$ and has degree $2^n$; for $n=5,6$ these degrees are respectively $32$ and $64$, both larger than $n^2/2$.  Hence $U$ is quasiequivalent to the natural module.  Twisting the quotient identification by the corresponding automorphism, we may take $U$ to be natural.  Therefore
\[
 |N|=q^{2n+1}.
\]
\end{proof}

\begin{lemma}\label{lem:Bstab}
Let $V$ be the natural module for $H=\OmegaG_{2n+1}(q)$ with $q$ odd and $n\ge4$, and let $0\ne v\in V$ be singular. Then
\[
 H_v=Q\rtimes L,
 \qquad |Q|=q^{2n-1},
 \qquad L\cong\OmegaG_{2n-1}(q),
\]
and
\[
 [H:H_v]=q^{2n}-1.
\]
In particular, $[H:H_v]$ is prime to $p$.
\end{lemma}

\begin{proof}
The stabilizer of the singular $1$-space $\langle v\rangle$ is the standard maximal parabolic of type $P_1$.  Specializing \cite[Proposition 4.1.20]{KL} to a singular $1$-space in odd dimension gives a unipotent radical of order $q^{2n-1}$ and a Levi component with an $\F_q^\times$ factor acting on $\langle v\rangle$ together with the orthogonal factor $\OmegaG_{2n-1}(q)$.  Passing from the $1$-space stabilizer to the vector stabilizer sets this scalar equal to $1$, giving
\[
 H_v=Q\rtimes\OmegaG_{2n-1}(q).
\]
The standard order formulas then give
\[
 \frac{|\OmegaG_{2n+1}(q)|}{q^{2n-1}|\OmegaG_{2n-1}(q)|}=q^{2n}-1.
\]
Equivalently, this is the number of nonzero singular vectors, in accordance with Witt transitivity.
\end{proof}

\begin{proposition}\label{prop:Bnonsplit}
Let $q=p^f$ be odd, $n\ge5$, and let
\[
 1\longrightarrow N\longrightarrow G\longrightarrow H=\OmegaG_{2n+1}(q)\longrightarrow1
\]
be a nonsplit extension in which $N$ is the natural module. Then there exists a faithful $\chi\in\Irr(G)$ such that
\[
 \chi(1)_p<|N|.
\]
\end{proposition}

\begin{proof}
Choose a nonzero singular vector $v\in N$.  Fix a nonprincipal additive character of $\F_q$ and use the nondegenerate symmetric bilinear form on $N$ to obtain the corresponding linear character $\lambda=\lambda_v\in\Irr(N)$.  This identification is $H$-equivariant, so
\[
 I_H(\lambda)=H_v.
\]
Put $T=I_G(\lambda)$ and $S=T/N$.  By Lemma~\ref{lem:Bstab},
\[
 S=Q\rtimes L,\qquad |Q|=q^{2n-1},\qquad L\cong\OmegaG_{2n-1}(q),
\]
and $[G:T]=q^{2n}-1$ is prime to $p$.

The character triple $(T,N,\lambda)$ determines a scalar factor set
\[
 \alpha\in H^2(S,\C^\times).
\]
Restrict it to $L$.  Since $n-1\ge4$, the generic multiplier formula in \cite[Theorem 5.1.4 and Table 5.1.A]{KL} gives
\[
 M(L)\cong C_2.
\]
If $\alpha|_L$ is trivial, let $\theta=1_L$.  If $\alpha|_L$ is nontrivial, it is the multiplier class afforded by the double universal cover $\Spin_{2n-1}(q)$.  Proposition~2.8 of \cite{HMII} constructs a faithful irreducible character of this double cover of degree prime to $p$.  It therefore affords an irreducible $\alpha|_L$-projective character $\theta$ of $L$ with
\[
 \theta(1)_p=1.
\]
Thus such a $\theta$ exists in either case.

Projectively induce $\theta$ from $L$ to $S$.  If
\[
 \theta^S=\sum_i m_i\tau_i
\]
is its decomposition into irreducible $\alpha$-projective characters, then
\[
 q^{2n-1}\theta(1)=\sum_i m_i\tau_i(1).
\]
The left side has $p$-adic valuation exactly $f(2n-1)$.  If every $\tau_i(1)$ had $p$-part strictly larger than $q^{2n-1}$, every term on the right would be divisible by $p^{f(2n-1)+1}$, a contradiction.  Hence some irreducible constituent $\tau$ satisfies
\[
 \tau(1)_p\le q^{2n-1}.
\]
By character-triple correspondence there is $\varphi\in\Irr(T\mid\lambda)$ with
\[
 \varphi(1)_p=\tau(1)_p.
\]
Clifford correspondence gives $\chi=\varphi^G\in\Irr(G)$, and Lemma~\ref{lem:Bstab} yields
\[
 \chi(1)_p=\varphi(1)_p\le q^{2n-1}<q^{2n+1}=|N|.
\]

It remains to prove faithfulness.  Put $K=\Ker(\chi)$.  Since $\chi$ lies over a nonprincipal character of the irreducible module $N$, the subgroup $K\cap N$ is a proper $H$-invariant subgroup of $N$, hence $K\cap N=1$.  Now $KN/N\lhd G/N=H$.  If $K\le N$, then $K=1$.  Otherwise simplicity gives $KN=G$, and then $K\cap N=1$ makes $K$ a complement to $N$, contradicting nonsplitting.  Thus $K=1$.
\end{proof}

\begin{proof}[Proof of Theorem~\ref{thm:Bmain}]
For $n=3$ the result is \cite[Proposition 4.3]{HMII}, so assume $n\ge4$.  Let $G,N,r$ satisfy the hypotheses of Conjecture B.  By Lemma~\ref{lem:HM}(2), we may assume $r=p$, the defining characteristic.  If $|N|>q^{n^2/2}$, Lemma~\ref{lem:HM}(1) finishes the proof.  Thus \eqref{eq:Bthreshold} holds.

For $n=4$, Proposition~\ref{prop:Bnatural} says that no such residual module exists.  Let $n\ge5$.  Proposition~\ref{prop:Bnatural} gives $|N|=q^{2n+1}$ and identifies $N$ with the natural module, up to twisting the quotient identification.  If the extension splits, Lemma~\ref{lem:HM}(3) gives a faithful character of $p'$-degree.  If it is nonsplit, Proposition~\ref{prop:Bnonsplit} gives the required faithful character.  Hence Conjecture B holds.
\end{proof}

\begin{proof}[Proof of Corollary~\ref{cor:Bmain}]
Apply Theorem~\ref{thm:Bmain} and \cite[Theorem~3.1]{HMII}.  Their only orthogonal exception is the even-dimensional family in odd characteristic, so it does not apply here.
\end{proof}

\section{Projective special unitary groups}\label{sec:unitary}
Let $H=\PSU_n(q)$, where $q=p^f$.  The cases $3\le n\le7$ are covered by \cite[Proposition~4.3]{HMII}, so only $n\ge8$ remains.

\begin{proposition}\label{prop:Unatural}
Let $n\ge8$, let $H=\PSU_n(q)$ be simple, and let $N$ be a faithful irreducible $\F_pH$-module satisfying
\[
 |N|\le \sqrt{|H|_p}=q^{n(n-1)/4}.
\]
Then $(n,q+1)=1$ and, after twisting the quotient identification by an automorphism of $H$, $N$ is the natural $n$-dimensional unitary module over $\F_{q^2}$, viewed over $\F_p$.  In particular $|N|=q^{2n}$.  For $n=8$ no such $N$ exists.
\end{proposition}

\begin{proof}
Inflate $N$ to $L=\SU_n(q)$, so $Z(L)$ acts trivially.  The smallest nontrivial absolutely irreducible defining-characteristic degree is $D=n$ \cite[Proposition 5.4.13 and Table 5.4.C]{KL}.  Apply \cite[Lemma~4.2]{GT}.

Alternative (v) gives $|N|\ge q^{2n^2}$, impossible.  In alternative (ii), an absolutely irreducible constituent has dimension at least $n^2$ and smallest field $\F_{q^{1/2}}$, so $|N|\ge q^{n^2/2}>q^{n(n-1)/4}$.  Alternative (iv) is only for ${}^3D_4$.

We now separate the two field-of-definition alternatives exactly as in the proof of \cite[Lemma~4.2]{GT}.  In alternative (i) the absolutely irreducible constituent is defined over the parameter field $\F_q$.  Two nontrivial restricted tensor factors would give dimension at least $n^2$ and hence $|N|\ge q^{n^2}$, impossible.  Thus there is one restricted constituent $W$, defined over $\F_q$, with $\dim W\le n(n-1)/4$.  The small-module list \cite[Proposition 5.4.11 and Table 5.4.A]{KL} leaves only the natural or dual highest weight in this range.  For the twisted group ${}^2A_{n-1}(q)$, however, these modules are not fixed by the defining graph--field Frobenius and their smallest absolute field of definition is $\F_{q^2}$; this is precisely the distinction recorded in \cite[Proposition 5.4.4(ii), Proposition 5.4.6(ii), and Corollary 2.10.10]{KL}.  Hence they cannot occur in alternative (i), whose constituent is $\F_q$-defined.

We are therefore in alternative (iii), the twisted-field case in Case~2 of \cite[Lemma~4.2]{GT}.  Here the constituent has smallest absolute field $\F_{q^2}$, and Lemma~\ref{lem:descent} gives
\[
 |N|=q^{2\dim W},\qquad \dim W\le \frac{n(n-1)}8.
\]
The same small-module list forces $W$ to be natural or dual up to graph and field twist.  Notice that this conclusion now uses the field convention in \cite[Lemma~4.2]{GT} explicitly: $q$ is the unitary parameter, while the natural module is defined over $\F_{q^2}$.  If $z=\zeta I\in Z(\SU_n(q))$, then on such a twist $z$ acts as $\zeta^{\pm p^j}$ for some $j$.  Since $|Z(\SU_n(q))|=(n,q+1)$ is prime to $p$, every nonidentity central element acts nontrivially.  But $Z(L)$ acts trivially on the inflated module $N$.  Thus $Z(L)=1$, equivalently $(n,q+1)=1$.  Therefore $H=L$ and $N$ is the natural additive module, so $|N|=q^{2n}$.  If $n=8$, this contradicts $q^{16}\le q^{14}$.
\end{proof}

\begin{lemma}\label{lem:Ustab}
Assume $n\ge9$ and $(n,q+1)=1$.  Let $V$ be the natural Hermitian space for $H=\SU_n(q)$, with Hermitian form $h$, and let $v$ be nonsingular.  Fix a nontrivial additive character $\nu$ of $\F_p$ and define
\[
 \lambda_v(x)=\nu\!\left(\operatorname{Tr}_{\F_{q^2}/\F_p}(h(v,x))\right)\qquad(x\in V).
\]
Then
\[
 I_H(\lambda_v)=H_v\cong\SU_{n-1}(q),
\]
and
\[
 [H:H_v]=q^{n-1}\bigl(q^n-(-1)^n\bigr).
\]
\end{lemma}

\begin{proof}
Nondegeneracy makes $v\mapsto\lambda_v$ injective and $H$-equivariant.  Since $v$ is nonsingular, $V=\langle v\rangle\perp v^\perp$, and restriction identifies its stabilizer with $\SU(v^\perp)\cong\SU_{n-1}(q)$.  The index follows from the standard order formula.
\end{proof}

\begin{proposition}\label{prop:Unonsplit}
Let $n\ge9$, $(n,q+1)=1$, and let
\[
 1\to N\to G\to\SU_n(q)\to1
\]
be a nonsplit perfect extension with $N$ the natural additive module.  Then there is a faithful $\chi\in\Irr(G)$ with $\chi(1)_p<|N|$.
\end{proposition}

\begin{proof}
Choose nonsingular $v$, put $\lambda=\lambda_v$ and $T=I_G(\lambda)$.  Then $T/N\cong\SU_{n-1}(q)$.  Since $n-1\ge8$, the group $\SU_{n-1}(q)$ is the full covering group of $\PSU_{n-1}(q)$ by \cite[Theorem 5.1.4 and Tables 5.1.A, 5.1.D]{KL}; hence it is centrally closed and
\[
 H^2(\SU_{n-1}(q),\C^\times)=1.
\]
Thus the factor set of the character triple $(T,N,\lambda)$ is trivial and $\lambda$ extends to $\widetilde\lambda\in\Irr(T)$.  Clifford correspondence gives
\[
 \chi=\widetilde\lambda^{\,G}\in\Irr(G),\qquad
 \chi(1)=q^{n-1}\bigl(q^n-(-1)^n\bigr),
\]
so $\chi(1)_p=q^{n-1}<q^{2n}=|N|$.

If $K=\Ker(\chi)$, then $K\cap N=1$ because $\chi$ lies nontrivially over the irreducible module $N$.  If $K\ne1$, simplicity of $G/N$ forces $KN=G$, making $K$ a complement to $N$, contrary to nonsplitting.  Hence $K=1$.
\end{proof}

\begin{proof}[Proof of Theorem~\ref{thm:Umain}]
The cases $3\le n\le7$ are in \cite[Proposition 4.3]{HMII}.  Let $n\ge8$ and let $G,N,r$ satisfy Conjecture B.  By Lemma~\ref{lem:HM}(2), assume $r=p$.  If $|N|>q^{n(n-1)/4}$, use Lemma~\ref{lem:HM}(1).  Otherwise Proposition~\ref{prop:Unatural} gives $(n,q+1)=1$ and $|N|=q^{2n}$, and excludes $n=8$.  Thus $n\ge9$.  The split case is Lemma~\ref{lem:HM}(3), while the nonsplit case is Proposition~\ref{prop:Unonsplit}.
\end{proof}

\begin{proof}[Proof of Corollary~\ref{cor:Umain}]
Apply Theorem~\ref{thm:Umain} and \cite[Theorem 3.1]{HMII}; the exceptional family in that implication is the even-dimensional orthogonal family in odd characteristic, not the unitary family.
\end{proof}

For a nondegenerate quadratic space $V$ of even dimension $2n$ over a field of odd order, let $CO(V)$ denote its conformal orthogonal group and let $\tau:CO(V)\to\F_q^\times$ be the conformal multiplier.  We write
\[
 CO(V)^0=\{g\in CO(V):\det(g)=\tau(g)^n\},
 \qquad P(C)=C/Z(C)\quad(C=CO(V)^0).
\]

\begin{lemma}\label{lem:projcent2}
Let $V$ be a nondegenerate quadratic space over $\F_q$ with $q$ odd and
$\dim V\ge6$, let $C=CO(V)^0$, and let $\bar t$ be the image of $t\in C$ in
$P(C)=C/Z(C)$.  Then $Z(C)=\F_q^\times I$, and
\[
 [C_{P(C)}(\bar t):\overline{C_C(t)}]\le 2.
\]
More precisely, if $gZ(C)$ centralizes $\bar t$, then
\[
 gtg^{-1}=ct\qquad(c\in\F_q^\times),
\]
and comparison of conformal multipliers gives $c^2=1$.  Consequently there are at most two cosets of $\overline{C_C(t)}$ in the
projective centralizer; the nontrivial coset, when it exists, is represented
precisely by an element conjugating $t$ to $-t$.
\end{lemma}
\begin{proof}
The scalar subgroup $\F_q^\times I$ is contained in $Z(C)$.  Conversely, an
element of $Z(C)$ centralizes $\SO(V)$; since $\dim V\ge6$, the natural
$\F_qSO(V)$-module is absolutely irreducible, so its centralizer in
$\GL(V)$ consists of scalars.  Hence $Z(C)=\F_q^\times I$.  Therefore
projective centralization means $gtg^{-1}=zt$ with $z=cI$ for some
$c\in\F_q^\times$.  The conformal multiplier is invariant under conjugacy, while
$\tau(ct)=c^2\tau(t)$.  Hence $c^2=1$.  If $t$ is scalar, then $C_{P(C)}(\bar t)=P(C)$ and
$\overline{C_C(t)}=P(C)$, so the index is $1$.  Assume that $t$ is not
scalar.  Every projective-centralizer element therefore has a representative
conjugating $t$ either to $t$ or to $-t$.  The first set is exactly
$\overline{C_C(t)}$.  If no representative conjugates $t$ to $-t$, we are
done.  Otherwise fix $g_0\in C$ with $g_0tg_0^{-1}=-t$.  If also
$gtg^{-1}=-t$, then
\[
 (g_0^{-1}g)t(g_0^{-1}g)^{-1}=t,
\]
so $g_0^{-1}g\in C_C(t)$.  Thus all representatives sending $t$ to $-t$
form a single coset of $C_C(t)$, and hence all their projective images form a
single coset of $\overline{C_C(t)}$.  Therefore the required index is at
most $2$.
\end{proof}

\begin{lemma}\label{lem:extfieldcoset}
Let $E=\F_{q^2}$, let $W$ be an even-dimensional nondegenerate quadratic
space over $E$ of type $\epsilon\in\{+,-\}$, and consider one of the
extension-field factors $O(W)$ occurring in the orthogonal
centralizer model.  Let $F$ denote the nontrivial $E/\F_q$ field
automorphism.  After an $E$-linear correction, $F$ induces an
$\F_q$-isometry normalizing the extension-field factor.  Its determinant
on the underlying $\F_q$-space is $+1$ for $\epsilon=+$ and $-1$ for
$\epsilon=-$.
\end{lemma}
\begin{proof}
Write $A$ for a Gram matrix of the $E$-quadratic form and $d=\det A$.
Choose $h\in \GL(W)$ with
\[
 h^T A h=A^{(q)}
\]
and put $\sigma=hF$.  In the extension-field centralizer model of \cite[Lemma~2.5]{Ng}
\cite[Lemma~2.5(i)]{Ng}, this corrected semilinear map is precisely an
$\F_q$-isometry of the ambient restriction-of-scalars orthogonal space
normalizing the factor $O(W)$.  Thus $\sigma$ is the required
$\F_q$-isometry.  Taking determinants gives
\[
 \det(h)^2=d^{q-1}.
\]
Since $q+1$ is even, raising this identity to the $(q+1)/2$ power gives
\[
 N_{E/\F_q}(\det h)=(\det h)^{q+1}
   =d^{(q^2-1)/2}.
\]
Thus the sign is independent of the choice of the correcting map $h$: it is
$+1$ precisely when the square class of $d$ is trivial and $-1$ precisely
when it is nontrivial.  To relate this to orthogonal type without suppressing
the usual discriminant convention, write $\dim_EW=2r$.  The type is
determined by the square class of $(-1)^r d$.  Since $E=\F_{q^2}$ and
$q$ is odd, $-1$ is a square in $E$; hence $(-1)^r d$ and $d$ have the
same square class.  Thus the two signs above are exactly the plus and minus
types, respectively.  If $h$ is replaced by
another correction, their quotient lies in $O(W)$ and has $E$-determinant
$\pm1$, whose norm to $\F_q$ is $1$, giving the same $\F_q$-determinant.
For completeness, the two determinant identities used here are intrinsic.
For an $E$-linear map $u$, restriction of scalars gives
\[
 \det_{\F_q}u=N_{E/\F_q}(\det_Eu).
\]
On one copy of $E$, choose $\alpha\in E$ with $\alpha^q=-\alpha$;
relative to the $\F_q$-basis $1,\alpha$, Frobenius has matrix
$\operatorname{diag}(1,-1)$ and hence determinant $-1$.  Therefore
$\det_{\F_q}F=(-1)^{\dim_EW}=1$, since $\dim_EW$ is even.  Consequently
\[
 \det_{\F_q}\sigma
 =N_{E/\F_q}(\det_Eh)\det_{\F_q}F
 =d^{(q^2-1)/2},
\]
which is $+1$ in plus type and $-1$ in minus type.
\end{proof}

We will also use the following order convention repeatedly.  For a semisimple element $t\in C=CO_{2n}^{\varepsilon}(q)^0$, in either the square- or nonsquare-multiplier cases of \cite[Lemma~2.5]{Ng}, the exact conformal centralizer satisfies
\[
 |C_C(t)|_{p'}=(q-1)|C_{\SO(V)}(t)|_{p'}.
\]
Since $Z(C)$ has order $q-1$ and is contained in $C_C(t)$, the exact projective image therefore has
\begin{equation}\label{eq:projectiveimageorder}
 \bigl|\overline{C_C(t)}\bigr|_{p'}
 =|C_{\SO(V)}(t)|_{p'}.
\end{equation}
The full projective centralizer has $p'$-part equal to the quantity in \eqref{eq:projectiveimageorder} or twice that quantity, according as a projective-centralizing coset outside the exact centralizer is absent or present.  Whenever we claim an exact value below, we indicate separately whether that second component occurs.

\begin{lemma}\label{lem:two-ppd}
Let $q$ be odd and let $n\ge8$ be even.  Let $t$ be a square-multiplier
semisimple element in $CO^+_{2n}(q)^0$.  Suppose that the exact centralizer
of $t$ has, in the decomposition \cite[Lemma~2.5(ii)]{Ng}, total rank
$n$, and that its order is divisible by primitive prime divisors $r_n$ and
$r_{n-1}$ of $q^n-1$ and $q^{n-1}-1$, respectively, while every odd prime
$r\ne p$ dividing the $p'$-part of the centralizer has multiplicative order at most $n$ modulo $q$, where $p$ is the defining characteristic.
Then the exact centralizer has full $\GL_n(q)$ type.
\end{lemma}
\begin{proof}
For odd $q$ and exponents at least $7$, the required primitive divisors
exist by Zsigmondy's theorem; the exceptional pair $(2,6)$ and the exponent
$2$ exception are irrelevant.

Suppose first that $r_n$ lies in an orthogonal factor of rank $k<n$.
If $k<n/2$, none of the factors $q^k\mp1$ and $q^{2i}-1$ ($i<k$) can contain
a prime of order $n$.  If $k=n/2$, only a minus factor, through
$q^{n/2}+1$, can contain $r_n$.  This same factor cannot contain $r_{n-1}$:
its factors $q^{2i}-1$ have $i<k=n/2$, so $n-1\nmid 2i$ (as $n-1$ is odd,
this would require $n-1\mid i$), while the terminal factor
$q^{n/2}+1$ cannot have a prime of order $n-1$.  The remaining rank $n/2$
is likewise too small to contain $r_{n-1}$.

If $k=n/2+1$, the orthogonal factor again cannot contain $r_{n-1}$.  Indeed,
for an internal factor $q^{2i}-1$ one has $i\le k-1=n/2<n-1$, whereas a
prime of order $n-1$ in the terminal factor $q^k\mp1$ would force
$n-1\mid k$ or $n-1\mid2k=n+2$; both are impossible for even $n\ge8$.
The complementary rank $n-k=n/2-1$ is also too small to contain a prime of
order $n-1$.  Finally, if $n/2+1<k<n$, the orthogonal factor contains a
primitive divisor of $q^{2(k-1)}-1$.  This divisor exists, since
$2(k-1)\ge n+2\ge10$, and has order $2(k-1)>n$, contrary to the hypothesis
on odd prime divisors.  Thus $r_n$ lies in an extension-field factor.

If $\GL_a(q^e)$ contains $r_n$, then $n\mid ej$ for some $j\le a$, while
$ea\le n$.  Hence $ea=n$ and $j=a$.  Since $ea=n$, this factor uses the entire available rank, so there are no
other noncentral factors in the decomposition \cite[Lemma~2.5]{Ng}.  Hence the same factor
must also contain $r_{n-1}$.  Thus $n-1\mid e\ell$ for some
$1\le \ell\le a=n/e$.  Since $e\mid n$, we have $(e,n-1)=1$, and therefore
$n-1\mid\ell$.  If $e\ge2$, however,
\[
 \ell\le n/e\le n/2<n-1,
\]
a contradiction.  Hence $e=1$ and $a=n$.

A unitary factor cannot replace this linear factor.  If $r_n$ occurs in a
minus-sign factor $q^{ej}-1$, then $n\mid ej$; together with
$ej\le ea\le n$ this forces $ej=ea=n$, so the factor already uses all
rank, and the primitive divisor of order $n-1$ has the wrong unitary
parity.  If $r_n$ occurs in a plus-sign factor $q^{ej}+1$, then
\[
 n\mid 2ej,\qquad n\nmid ej.
\]
Since $ej\le ea\le n$, this implies $ej=n/2$.  This factor has rank
cost at least $n/2$, so the complementary rank is at most $n/2$.  A
separate factor containing the primitive divisor $r_{n-1}$ would have
rank cost at least $n-1$: a linear factor needs $e'a'\ge n-1$; a
unitary plus-sign factor cannot contain a prime of odd order $n-1$, while
a unitary minus-sign factor needs $e'a'\ge n-1$; and an orthogonal
factor likewise needs rank at least $n-1$.  Since $n-1>n/2$, this is
impossible.  Hence $r_{n-1}$ must occur in the same unitary factor.
Supplying it there requires a minus-sign unitary factor $q^{ej'}-1$ with
$j'$ even.  As $e\mid n/2$ and $(e,n-1)=1$, divisibility
$n-1\mid ej'$ forces $j'\ge n-1$; the rank bound then forces $e=1$,
but $j'=n-1$ is odd, a contradiction.  Hence the only possibility is
$\GL_n(q)$.
\end{proof}

\section{Elementary-abelian extensions of even-dimensional orthogonal groups in odd characteristic}\label{sec:DallB}
The elementary-abelian case admits a uniform treatment, including ranks five and six.

\begin{proposition}\label{prop:DallB}
Let $q=p^f$ be odd, let $n\ge5$, let $\varepsilon\in\{+,-\}$, and suppose that
\[
 S=\POmega_{2n}^{\varepsilon}(q)
\]
is simple.  Then Conjecture B holds for $S$.
\end{proposition}

\begin{proof}
Let $G,N,r$ satisfy the hypotheses of Conjecture B.  By Lemma~\ref{lem:HM}(2), we may assume $r=p$, and by Lemma~\ref{lem:HM}(1) we may assume
\begin{equation}\label{eq:DallBbound}
 |N|\le \sqrt{|S|_p}=q^{t},\qquad t=\frac{n(n-1)}2.
\end{equation}
Apply \cite[Lemma~4.2]{GT}.  The smallest nontrivial absolutely irreducible defining-characteristic projective degree is $D=2n$.  Alternative (v) gives $|N|\ge q^{2D^2}$, and alternative (ii) gives
\[
 |N|\ge(q^{1/2})^{D^2}=q^{2n^2};
\]
both contradict \eqref{eq:DallBbound}.  In alternative (i), two nontrivial restricted tensor factors give $|N|\ge q^{D^2}$, again impossible.  Hence there is one restricted constituent $W$, defined over $\F_q$, with
\[
 |N|=q^{\dim W},\qquad \dim W\le t.
\]
By \cite[Proposition~5.4.11 and Table~5.4.A]{KL}, the only projective module in this range is the natural $2n$-dimensional orthogonal module, up to field and automorphism twist.  Indeed, the next exterior-square entry has dimension
\[
 n(2n-1)-\gcd(2,n)>t,
\]
and every spin entry has dimension at least $2^{n-1}>t$ for $n\ge5$.  Thus $N$ is of natural-module type.

For the twisted group ${}^2D_n(q)$ one must also exclude alternative (iii) of \cite[Lemma~4.2]{GT}.  There the restricted constituent has smallest field $\F_{q^2}$, so \eqref{eq:DallBbound} gives $\dim W\le t/2$.  The same small-module list again forces $W$ to be of natural type, whereas in Case~2 of the proof of \cite[Lemma 4.2]{GT} alternative (iii) requires $W\not\cong W^{\tau_0}$ for the graph automorphism $\tau_0$.  The natural $D_n$ highest weight, and every Frobenius twist of it, is fixed by $\tau_0$.  Hence alternative (iii) is impossible.

We are therefore left only with the natural module.  If it does not factor through $S$, there is nothing to prove.  Suppose that it does.  Then $-I\notin\OmegaG_{2n}^{\varepsilon}(q)$ and $S=\OmegaG_{2n}^{\varepsilon}(q)$ in its natural representation.  After twisting the quotient identification we may take $N=V$ additively, so $|N|=q^{2n}$.

The split extension is covered by Lemma~\ref{lem:HM}(3).  Assume it is nonsplit.  Use the polar form and the trace pairing to identify $V$ with $\Irr(N)$, and choose a nonsingular vector $v$.  For the corresponding $\lambda_v\in\Irr(N)$,
\[
 L:=I_S(\lambda_v)=S_v\cong\OmegaG_{2n-1}(q),
 \qquad [S:L]_p=q^{n-1}.
\]
Put $T=I_G(\lambda_v)$.  The character triple $(T,N,\lambda_v)$ determines a factor set on $L$.  Since $n-1\ge4$, its Schur multiplier has order at most $2$.  If the factor set is trivial, use the linear projective character of $L$; if it is nontrivial, Proposition~2.8 of \cite{HMII} supplies a faithful $p'$-degree character of $\Spin_{2n-1}(q)$, hence an irreducible projective character $\theta$ of $L$ with $\theta(1)_p=1$.  Character-triple correspondence followed by Clifford induction gives a character $\chi\in\Irr(G)$ with
\[
 \chi(1)_p=q^{n-1}<q^{2n}=|N|.
\]
As usual, $\Ker(\chi)\cap N=1$ by irreducibility of $N$; if $\Ker(\chi)\ne1$, its image in the simple quotient is all of $S$, producing a complement to $N$ and contradicting nonsplitting.  Thus $\chi$ is faithful.
\end{proof}

\begin{remark}
For $n=4$, Conjecture B is already proved in \cite[Proposition~4.3]{HMII}.  Hence Conjecture B holds for every simple $\POmega_{2n}^{\varepsilon}(q)$ with $q$ odd and $n\ge4$.
\end{remark}

\section{Plus-type even-dimensional orthogonal groups of odd rank in odd characteristic}\label{sec:Doddcover}
The general implication from Conjecture B to the codegree isomorphism conjecture excludes even-dimensional orthogonal groups in odd characteristic because the quasisimple central-cover argument is missing in this family \cite[Theorem~3.1 and the discussion after Proposition~2.9]{HMII}. For the subfamily below we prove both the central-cover case and Conjecture B, and hence obtain the full codegree isomorphism statement.

\begin{theorem}\label{thm:Doddcover}
Let $n\ge7$ be odd, let $q=p^f\equiv3\pmod4$, and set
\[
 S=\POmega^+_{2n}(q),\qquad X=\Spin^+_{2n}(q).
\]
Then $Z(X)$ has order $2$, $X/Z(X)\cong S$, and there exists a faithful character $\chi\in\Irr(X)$ such that
\[
 \chi(1)=\prod_{i=1}^{n-1}(q^i+1)
 \qquad\text{and}\qquad
 \frac{\chi(1)}{2}\notin\operatorname{cd}(S).
\]
In particular, the conclusion of \cite[Theorem D]{HMII} holds for $S$.
\end{theorem}

\begin{proof}
Since $n$ is odd and $q\equiv3\pmod4$, the standard multiplier formula gives
\[
 |Z(X)|=(4,q^n-1)=2.
\]
The dual group of $X$ is
\[
 G^*=P(CO^+_{2n}(q)^0).
\]
For odd $q$, Tiep and Zalesskii \cite[Remark~7.3 and Lemma~7.4]{TZ} describe
\[
 CO^+_{2n}(q)^0=\{g\in CO^+_{2n}(q):\det(g)=\tau(g)^n\},
\]
where $\tau$ is the conformal multiplier.  Write the natural plus-type space as a hyperbolic sum
\[
 V=U\oplus U^*,\qquad \dim U=n.
\]
Choose $a,b\in\F_q^\times$ with $ab$ a nonsquare and $a\ne b$, and put
\[
 s=\operatorname{diag}(aI_U,bI_{U^*}).
\]
(For $q=3$ take $a=1$, $b=-1$.) Then $\tau(s)=ab$ and $\det(s)=(ab)^n$, so $s\in CO^+_{2n}(q)^0$. Let $s^*$ denote its image in $G^*$. The exact conformal centralizer gives, modulo scalars, a $\GL_n(q)$-centralizer. There is no additional projective element interchanging $U$ and $U^*$ inside $CO^+_{2n}(q)^0$: if
\[
 g=\begin{pmatrix}0&A\\ B&0\end{pmatrix}
\]
is a conformal isometry interchanging the two totally singular summands, then the conformal relation gives $\det(AB)=\tau(g)^n$, and hence
\[
 \det(g)=(-1)^n\det(AB)=-\tau(g)^n
\]
because $n$ is odd. This contradicts the defining relation $\det(g)=\tau(g)^n$ for $CO^+_{2n}(q)^0$. Therefore
\[
 C_{G^*}(s^*)\cong \GL_n(q).
\]
Let $\mathcal E(X,s^*)$ denote the Lusztig series of $X$ associated with $s^*$.  By Lusztig--Jordan decomposition there is a semisimple character $\chi\in\mathcal E(X,s^*)$ with
\[
 \chi(1)=|G^*:C_{G^*}(s^*)|_{p'}
 =\frac{(q^n-1)\prod_{i=1}^{n-1}(q^{2i}-1)}{\prod_{i=1}^{n}(q^i-1)}
 =\prod_{i=1}^{n-1}(q^i+1).
\]
Lemma~2.2 of \cite{HMII} gives the needed central-character criterion: since
\[
 |Z(X)|=|G^*/(G^*)'|=2,
\]
all characters in $\mathcal E(X,t^*)$ are trivial on $Z(X)$ if and only if $t^*\in(G^*)'$; if $t^*\notin(G^*)'$, they all lie over the unique nontrivial character of $Z(X)$ \cite[Lemma 2.2]{HMII}. In the present case the square class of the conformal multiplier induces the quotient map
\[
 G^*\longrightarrow \F_q^\times/(\F_q^\times)^2,
 \qquad \bar g\longmapsto \tau(g)(\F_q^\times)^2.
\]
Its kernel is $(G^*)'=\POmega^+_{2n}(q)$ in our parity. Since $\tau(s)$ is nonsquare, $s^*\notin(G^*)'$. Hence $\chi$ is nontrivial on the unique central involution of $X$, and therefore is faithful.

It remains to prove that $\chi(1)/2$ is not a character degree of $S$.  Suppose otherwise, and choose $\psi\in\Irr(S)$ with
\[
 \psi(1)=\frac12\prod_{i=1}^{n-1}(q^i+1).
\]
Inflate $\psi$ to $X$. Its degree is prime to $p$, so let $((t^*),\eta)$ be its Jordan decomposition. Then $t^*\in(G^*)'$ by \cite[Lemma 2.2]{HMII}, and the degree formula in \cite[Theorem~1.4]{Ng} gives
\[
 \psi(1)=|G^*:C_{G^*}(t^*)|_{p'}\,\eta(1),
\]
where $\eta$ is a unipotent character of $C_{G^*}(t^*)$ and $p\nmid\eta(1)$; see \cite[Section 2.2, formula (2.1)]{Ng}. Consequently
\begin{equation}\label{eq:Dodd-centralizer}
 |C_{G^*}(t^*)|_{p'}=2\eta(1)\,|\GL_n(q)|_{p'}.
\end{equation}

We now use the description of semisimple centralizers in $CO^+_{2n}(q)^0$ from \cite[Lemma~2.5]{Ng}. Since $t^*\in(G^*)'$, its multiplier is a square.  After choosing a lift $t\in CO^+_{2n}(q)^0$, the exact centralizer in the orthogonal group has the form
\[
 O_{2k}^{\delta_1}(q)\times O_{2(m-k)}^{\delta_2}(q)
 \times\prod_j \GL_{a_j}^{\alpha_j}(q^{e_j}),
 \qquad
 m+\sum_j e_ja_j=n,
\]
where $\GL^+=\GL$ and $\GL^-=\GU$.

Let $r_n$ be a primitive prime divisor of $q^n-1$.  Such a prime exists by Zsigmondy's theorem because $q$ is odd and $n\ge7$.  Equation \eqref{eq:Dodd-centralizer} forces $r_n$ to divide the projective centralizer. Since $r_n$ is odd and Lemma~\ref{lem:projcent2} shows that a projective centralizer can exceed the image of the exact conformal centralizer by at most a factor $2$, $r_n$ already divides the exact centralizer described in \cite[Lemma~2.5]{Ng}.  A proper orthogonal factor cannot contain $r_n$.  Indeed, if $r_n\mid q^{2i}-1$ with $i<n$, then $n\mid2i$; since $n$ is odd this gives $n\mid i$, impossible.  If $r_n\mid q^k-1$, then $n\mid k$, again impossible for $k<n$.  Finally, if $r_n\mid q^k+1$, then $q^{2k}\equiv1\pmod{r_n}$, so $n\mid2k$. Since $n$ is odd, $n\mid k$, whence $q^k\equiv1\pmod{r_n}$, contradicting $q^k\equiv-1\pmod{r_n}$. Thus the plus terminal factor cannot contain $r_n$ at all.  If an orthogonal factor has rank $n$, the whole exact centralizer is the full orthogonal group, which would give $C_{G^*}(t^*)=G^*$ and contradict \eqref{eq:Dodd-centralizer}.  Hence $r_n$ must occur in one extension-field factor.

A unitary factor $\GU_a(q^e)$ cannot contain $r_n$.  Indeed, if $r_n\mid q^{ej}-(-1)^j$, then for odd $j$ we would have $q^{ej}\equiv-1\pmod{r_n}$, impossible because $\operatorname{ord}_{r_n}(q)=n$ is odd; for even $j$ we obtain $n\mid ej$.  The rank inequality $ea\le n$ then forces $ej=ea=n$, so $j=a$, contradicting the parity of $a$ because $ea=n$ is odd.  Thus the relevant factor is $\GL_a(q^e)$, and
\[
 n\mid ej\quad\text{for some }j\le a,
 \qquad ea\le n.
\]
It follows that $ea=n$ and $j=a$; in particular this one factor consumes the entire rank.

Now let $r_{n-1}$ be a primitive prime divisor of $q^{n-1}-1$.  Again $r_{n-1}$ is odd, so the same exact/projective observation and \eqref{eq:Dodd-centralizer} force $r_{n-1}$ into $\GL_a(q^e)$, so
\[
 n-1\mid ej'
\]
for some $j'\le a$.  Since $e\mid n$, we have $(e,n-1)=1$, and hence $n-1\mid j'$.  If $e>1$, then $e\ge3$ and $a=n/e<n-1$, a contradiction.  Therefore
\[
 e=1,\qquad a=n.
\]
Thus, after quotienting the scalar conformal factor, the image of the exact centralizer has $p'$-order $|\GL_n(q)|_{p'}$. As above, an element that projectively centralizes this full $\GL_n$-type element but is not in the exact centralizer would have to interchange the two $n$-dimensional totally singular summands. The block-matrix determinant calculation gives
\[
 \det(g)=(-1)^n\tau(g)^n=-\tau(g)^n,
\]
so no such element lies in $CO^+_{2n}(q)^0$. Hence in fact
\[
 |C_{G^*}(t^*)|_{p'}=|\GL_n(q)|_{p'}.
\]
This contradicts \eqref{eq:Dodd-centralizer}, since $2\eta(1)\ge2$. Therefore $\chi(1)/2\notin\operatorname{cd}(S)$.
\end{proof}

\begin{proposition}\label{prop:DoddB}
Let $n\ge7$ be odd, let $q=p^f\equiv3\pmod4$, and let
\[
 H=\POmega^+_{2n}(q).
\]
Then Conjecture B holds for $H$.
\end{proposition}

\begin{proof}
Let $G$ be perfect and let $N\lhd G$ be an elementary abelian $r$-group such that $G/N\cong H$ and $N$ is a faithful irreducible $H$-module. By Lemma~\ref{lem:HM}(2), we may assume that $r=p$, the defining characteristic. By Lemma~\ref{lem:HM}(1), we may also assume
\begin{equation}\label{eq:DoddNbound}
 |N|\le \sqrt{|H|_p}=q^{n(n-1)/2}.
\end{equation}

We first determine $N$. Since $n$ is odd and $q\equiv3\pmod4$, Propositions~2.5.10 and~2.5.13 of \cite{KL} show that $-I\notin\Omega^+_{2n}(q)$: indeed, for a plus-type $2n$-space the discriminant is nonsquare because $n(q-1)/2$ is odd, and $-I$ lies in $\Omega$ exactly when this discriminant is square. Hence
\[
 H=\POmega^+_{2n}(q)=\Omega^+_{2n}(q),
\]
and the natural $2n$-dimensional orthogonal module is an honest faithful $H$-module; see \cite[Propositions 2.5.10 and 2.5.13]{KL}.

Apply \cite[Lemma~4.2]{GT} to the irreducible $\F_pH$-module $N$. The smallest nontrivial absolutely irreducible defining-characteristic degree is $D=2n$. Since type $D_n$ is untwisted, only alternatives (i), (ii), and (v) of that lemma can occur. Alternative (v) gives $|N|\ge q^{2D^2}$, contrary to \eqref{eq:DoddNbound}. In alternative (ii), an absolute constituent is a tensor product of two nontrivial restricted modules and its smallest field is $\F_{q^{1/2}}$; hence
\[
 |N|\ge (q^{1/2})^{D^2}=q^{2n^2}>q^{n(n-1)/2},
\]
again impossible. In alternative (i), two nontrivial restricted tensor factors similarly give $|N|\ge q^{D^2}$, so there is a single restricted constituent $W$, defined over $\F_q$, and
\[
 |N|=q^{\dim W},\qquad \dim W\le \frac{n(n-1)}2.
\]
The small-module result \cite[Proposition~5.4.11 and Table~5.4.A]{KL} now forces $W$ to be quasiequivalent to the natural module. Indeed, the exterior-square entry has dimension at least $n(2n-1)-2$, and every spin entry has dimension at least $2^{n-1}$; both exceed $n(n-1)/2$ for $n\ge7$. After twisting the quotient identification by the corresponding automorphism of $H$, we may therefore assume
\begin{equation}\label{eq:Doddnatural}
 N=V\quad\text{additively},\qquad |N|=q^{2n},
\end{equation}
where $V$ is the natural orthogonal space.

If the extension splits, Lemma~\ref{lem:HM}(3) finishes the proof. Assume that it is nonsplit. Let $B$ be the polar form of the quadratic form on $V$, fix a nonprincipal additive character $\nu$ of $\F_p$, and for $v\in V$ put
\[
 \lambda_v(x)=\nu\!\left(\operatorname{Tr}_{\F_q/\F_p} B(v,x)\right).
\]
The nondegeneracy of $B$ and of the trace pairing gives an $H$-equivariant isomorphism $V\cong\Irr(N)$. Choose a nonsingular $v$. Then
\[
 I_H(\lambda_v)=H_v.
\]
Since $V=\langle v\rangle\perp v^\perp$, restriction to $v^\perp$ identifies
\[
 L:=H_v\cong\OmegaG_{2n-1}(q).
\]
The spinor norm of an isometry fixing $v$ is the spinor norm of its restriction to $v^\perp$, so this is an equality with $\OmegaG_{2n-1}(q)$, not merely with $\SO_{2n-1}(q)$. From the standard order formulas,
\begin{equation}\label{eq:Doddindex}
 [H:L]_p=q^{n-1}.
\end{equation}

Put $T=I_G(\lambda_v)$, so $T/N\cong L$. The character triple $(T,N,\lambda_v)$ determines a factor set $\alpha\in H^2(L,\C^\times)$. In our range the Schur multiplier of $L=\OmegaG_{2n-1}(q)$ has order $2$. If $\alpha$ is trivial, choose the degree-one $\alpha$-projective character. If $\alpha$ is nontrivial, Proposition~2.8 of \cite{HMII} constructs a faithful irreducible character of the double cover $\Spin_{2n-1}(q)$ whose degree is prime to $p$; it affords an irreducible $\alpha$-projective character $\theta$ of $L$ with
\[
 \theta(1)_p=1.
\]
Thus in either case character-triple correspondence gives $\varphi\in\Irr(T\mid\lambda_v)$ with $\varphi(1)_p=1$. Clifford correspondence yields
\[
 \chi=\varphi^G\in\Irr(G),
\]
and by \eqref{eq:Doddindex},
\[
 \chi(1)_p=[G:T]_p\varphi(1)_p=q^{n-1}<q^{2n}=|N|.
\]

Finally $\chi$ is faithful. Since $\lambda_v\ne1_N$, the subgroup $\Ker(\chi)\cap N$ is a proper $H$-invariant subgroup of the irreducible module $N$, hence is trivial. If $\Ker(\chi)\ne1$, its image in the simple quotient $G/N$ is all of $H$, so $G=N\Ker(\chi)$ with $N\cap\Ker(\chi)=1$, producing a complement to $N$ and contradicting nonsplitting. Therefore $\Ker(\chi)=1$.
\end{proof}

\begin{theorem}\label{thm:Doddcodegree}
Let $n\ge7$ be odd and $q\equiv3\pmod4$ an odd prime power. Set
\[
 H=\POmega^+_{2n}(q).
\]
If $G$ is a nontrivial finite group and $\operatorname{cod}(G)\subseteq\operatorname{cod}(H)$, then $G\cong H$.
\end{theorem}

\begin{proof}
Apply Proposition~\ref{prop:assemblyreduction}.  Its elementary-abelian hypothesis is Proposition~\ref{prop:DoddB}, and its prime-central-cover hypothesis is Theorem~\ref{thm:Doddcover}.  Hence no counterexample exists.
\end{proof}

\subsection{Even rank in the same congruence class}\label{subsec:Deven}
We next take $n$ even.  The elementary-abelian case becomes shorter, whereas the central case has three double covers.

\begin{proposition}\label{prop:DevenB}
Let $n\ge8$ be even, let $q=p^f\equiv3\pmod4$, and put
\[
 H=\POmega^+_{2n}(q).
\]
Then Conjecture B holds for $H$.
\end{proposition}

\begin{proof}
This is the corresponding specialization of Proposition~\ref{prop:DallB}, proved uniformly for both signs, every odd $q$, and every $n\ge5$.  We state it separately here so that the prime-cover argument below can be read without referring back to the uniform formulation.
\end{proof}

We first need the following elementary Clifford-theory observation.

\begin{lemma}\label{lem:smallclifford}
Let $K\lhd L$ with $L/K$ a $2$-group of order at most $4$, and let
$\vartheta\in\Irr(K)$ be linear.  If $\xi\in\Irr(L\mid\vartheta)$, then
\[
 \xi(1)\in\{1,2,4\}.
\]
More precisely, if $m=[L:I_L(\vartheta)]$ and
$J=I_L(\vartheta)/K$, then
\[
 \xi(1)=m d,
\]
where $d$ is the degree of an irreducible projective representation of $J$;
hence $d^2\le |J|$ and $m|J|=[L:K]\le4$.
\end{lemma}

\begin{proof}
Clifford correspondence reduces $\xi$ to an irreducible character above
$\vartheta$ in $I_L(\vartheta)$.  Since $\vartheta$ is linear, the character
triple gives an irreducible module of degree $d$ for a twisted group algebra
$\mathbb C^\alpha[J]$.  This algebra has dimension $|J|$, so the square of
the dimension of each simple module is at most $|J|$; thus $d^2\le|J|$.
If $m=4$, then $J=1$ and $\xi(1)=4$.  If $m=2$, then $|J|\le2$; a twisted
group algebra of a cyclic group of order at most $2$ has only linear simple
modules, so $\xi(1)=2$.  If $m=1$, then $|J|\le4$ and $d^2\le4$, whence the
positive integer $d$ is $1$ or $2$.  The assertion follows.  Notice that
when $J\cong C_2\times C_2$ a nontrivial factor set may indeed give $d=2$;
this possibility is included.
\end{proof}

We need one elementary observation on the unipotent factor in Jordan decomposition.

\begin{lemma}\label{lem:Deveneta}
Let $q$ be odd, let $\varepsilon\in\{+,-\}$, and let $t^*\in P(CO^\varepsilon_{2n}(q)^0)$ have square conformal multiplier.  Let $\eta$ be the unipotent character of $C_{G^*}(t^*)$ occurring in Jordan decomposition.  If $p\nmid\eta(1)$, where $p$ is the defining characteristic, then
\[
 \eta(1)\in\{1,2,4\}.
\]
In particular, $\eta(1)$ is a power of $2$.
\end{lemma}

\begin{proof}
Put $C=C_{G^*}(t^*)$.  We use the definition of unipotent characters of a possibly disconnected semisimple centralizer that is used in the Jordan decomposition.  It is recalled in \cite[Section~2, the paragraph following (2)]{HMII}, from Digne--Michel, that when $C_{\mathbf G^*}(t^*)$ is disconnected, its unipotent characters are precisely the irreducible characters of the finite centralizer lying over unipotent characters of the finite points of the connected component; see the cited discussion.  Consequently every constituent of the restriction of $\eta$ to the finite connected component is unipotent.  The primary decomposition \cite[Section~2.3 and Lemma~2.5]{Ng} identifies that connected component, up to its central torus and central quotients, with a product of $\GL_a(q^e)$, $\GU_a(q^e)$, and $\SO$ factors.  Proposition~2.1 of \cite{Ng} (Digne--Michel, Proposition~13.20) passes unipotent characters through the central isogenies and central quotients \cite[Proposition~2.1]{Ng}.

Let $C^\circ$ denote the image in $C$ of this product of connected finite factors together with the central torus.  This subgroup is normal in $C$: every element of the exact centralizer normalizes the primary decomposition, while a possible projective element sending a lift $t$ to $-t$ merely permutes the primary factors and hence normalizes their connected product.  Let $\theta$ be a constituent of $\eta|_{C^\circ}$.  By the preceding definition and the primary-product description, $\theta$ is, up to central linear characters and central quotient, a tensor product of unipotent characters of the connected factors.  Clifford theory gives
\[
 \eta(1)=e[C:I_C(\theta)]\theta(1)
\]
for some positive integer $e$, so $p\nmid\eta(1)$ implies $p\nmid\theta(1)$.  Malle \cite[Theorem~6.8]{Mal} shows that the trivial character is the unique unipotent character of $p'$-degree for a simple finite reductive group, apart from the stated small exceptions in types $B_n,C_n,G_2,F_4$ at $q=2$ and $G_2(3)$.  None of those exceptions occurs here: the connected semisimple factors have type $A$ (including ${}^2A$) or type $D$, with $D_2=A_1A_1$ and $D_3=A_3$, and $q$ is odd.  Hence every connected semisimple factor contributes degree $1$, while central tori contribute only linear characters.  Thus every constituent $\theta$ of $\eta|_{C^\circ}$ is linear, so $\theta(1)=1$.  Any nonlinear degree of $\eta$ can therefore arise only from Clifford theory across the $2$-group $C/C^\circ$.

We next determine the component bound.  The decomposition \cite[Lemma~2.5]{Ng} has orthogonal factors only on the two $\pm1$-primary spaces; all non-$\pm1$ primary factors are $\GL$ or $\GU$ and are connected.  If both $V_+$ and $V_-$ are nonzero, the determinant-one condition couples their determinant signs, so
\[
 S\bigl(O(V_+)\times O(V_-)\bigr)/(\SO(V_+)\times \SO(V_-))
\]
has order at most $2$.  If exactly one of $V_+,V_-$ is nonzero, determinant one forces that orthogonal factor into $\SO$, and if both vanish the quotient is trivial.  Hence the exact finite centralizer has index at most $2$ over $C^\circ$.  By Lemma~\ref{lem:projcent2}, passage to the projective centralizer contributes at most one further coset, represented by an element conjugating a lift $t$ to $-t$.  Central quotients cannot increase this index.  Consequently
\[
 [C:C^\circ]\le4,
\]
and $C/C^\circ$ is a $2$-group.

Now $\theta$ is linear and $C^\circ\lhd C$ has index at most $4$.
Lemma~\ref{lem:smallclifford}, applied with $(K,L,\vartheta,\xi)=(C^\circ,C,\theta,\eta)$,
gives directly
\[
 \eta(1)\in\{1,2,4\}.
\]
This includes the possible degree-$2$ irreducible projective representation
of a Klein-four inertia quotient arising from a nontrivial factor set.
\end{proof}

\begin{theorem}\label{thm:Devencover}
Let $n\ge8$ be even, let $q=p^f\equiv3\pmod4$, and set
\[
 S=\POmega^+_{2n}(q),\qquad \widetilde S=\Spin^+_{2n}(q).
\]
For every subgroup $Z_0<Z(\widetilde S)$ of order $2$, the double cover
\[
 X=\widetilde S/Z_0,\qquad |Z(X)|=2,\qquad X/Z(X)\cong S,
\]
possesses a faithful $\chi\in\Irr(X)$ such that
\[
 \frac{\chi(1)}2\notin\operatorname{cd}(S).
\]
Thus the conclusion of \cite[Theorem D]{HMII} holds for $S$.
\end{theorem}

\begin{proof}
The standard multiplier formula gives
\[
 Z(\widetilde S)\cong C_2\times C_2.
\]
The dual group is
\[
 G^*=P(CO^+_{2n}(q)^0),
\]
and $G^*/(G^*)'$ has order $4$.  Lemma~2.2 of \cite{HMII} identifies the four cosets of $(G^*)'$ with the four central characters of $Z(\widetilde S)$.  The square class of the conformal multiplier gives a quotient of this abelianization of order $2$.  Hence among the three nontrivial cosets there is one square-multiplier coset and two nonsquare-multiplier cosets.  The type $D_n$ graph automorphism exchanges the two half-spin nodes, and therefore exchanges the latter two central characters; cf. \cite[Proposition 5.4.2(ii)]{KL}.  It is consequently enough to treat one square-multiplier nontrivial coset and one nonsquare-multiplier coset.

\smallskip
\noindent\emph{The square-multiplier coset.}
Since $q\equiv3\pmod4$, choose a hyperbolic $2$-space $W$ and let $s$ act as $-I$ on $W$ and as $I$ on $W^\perp$.  The spinor norm of $-I_W$ is the discriminant square class of $W$, namely the class of $-1$, and is therefore nonsquare.  Hence $s^*$ lies in the nontrivial square-multiplier coset.  This is the square-multiplier case $m=n$, $k=1$ in \cite[Lemma~2.5(ii)]{Ng}.  The corresponding semisimple character has degree
\begin{equation}\label{eq:DevenOmegaDegree}
 d_0:=\frac{(q^n-1)(q^{n-1}+1)}{2(q-1)}.
\end{equation}
By the central-character criterion \cite[Lemma 2.2]{HMII}, it descends faithfully to the double cover associated with this coset.

Put
\[
 d=\frac{d_0}{2}=\frac{(q^n-1)(q^{n-1}+1)}{4(q-1)}.
\]
Since $n\ge8$, we have $d<q^{4n-10}$.  If $d\in\operatorname{cd}(S)$, its inflation to $\widetilde S$ is an irreducible character of degree $d$, so Theorem~1.4 of \cite{Ng} forces $d$ to equal one of the following degrees (with $\beta=\pm1$):
\[
 1,\quad
 \frac{(q^n-1)(q^{n-1}+q)}{q^2-1},\quad
 \frac{q^{2n}-q^2}{q^2-1},\quad
 \frac{(q^n-1)(q^{n-1}-1)}{2(q+1)},
\]
\[
 \frac{(q^n-1)(q^{n-1}+1)}{2(q-1)},\qquad
 \frac{(q^n-1)(q^{n-1}+\beta)}{q-\beta}.
\]
Write $x=q^{n-1}$.  Equality with the second degree or with the fourth gives
\[
 (q-3)x=3q-1,
\]
which is impossible (also when $q=3$).  Equality with the third gives
\[
 q(3q-1)x=4q^2-q-1,
\]
again impossible.  Equality with the $\beta=-1$ degree gives
\[
 (3q-5)x=5q-3,
\]
and the remaining displayed possibilities are respectively $1$, $2d$, and $4d$.  Hence $d\notin\operatorname{cd}(S)$.

\smallskip
\noindent\emph{A nonsquare-multiplier coset, $q>3$.}
Choose a nonsquare $\tau\ne-1$ and put
\[
 s=\operatorname{diag}(I_U,\tau I_{U^*})
\]
for a hyperbolic decomposition $V=U\oplus U^*$ with $\dim U=n$.  Since $n$ is even, $s\in CO^+_{2n}(q)^0$.  As $\tau\ne-1$, no projective centralizer element interchanges $U$ and $U^*$, and
\[
 C_{G^*}(s^*)\cong \GL_n(q).
\]
Thus the associated semisimple character has degree
\begin{equation}\label{eq:DevenHalfDegree}
 d_1:=\prod_{i=1}^{n-1}(q^i+1).
\end{equation}
It lies over one of the two nonsquare-multiplier central characters and hence gives a faithful character of the corresponding double cover.

Suppose that $d_1/2=\psi(1)$ for some $\psi\in\Irr(S)$, and inflate $\psi$ to $\widetilde S$.  If $((t^*),\eta)$ is its Jordan decomposition, then $t^*\in(G^*)'$ by \cite[Lemma 2.2]{HMII}, and
\begin{equation}\label{eq:DevenCentralizer}
 |C_{G^*}(t^*)|_{p'}=2\eta(1)|\GL_n(q)|_{p'},
 \qquad \eta(1)\in\{1,2,4\}
\end{equation}
by Lemma~\ref{lem:Deveneta}.

Let $r_n$ and $r_{n-1}$ be primitive prime divisors of $q^n-1$ and $q^{n-1}-1$.  They exist by Zsigmondy's theorem and are odd.  Equation~\eqref{eq:DevenCentralizer} forces both into the exact semisimple centralizer, since Lemma~\ref{lem:projcent2} and $\eta(1)\in\{1,2,4\}$ add only powers of $2$.  The same fact $\eta(1)\in\{1,2,4\}$ is essential for the next hypothesis: it contributes no odd prime at all.  Hence every odd prime $r\ne p$ dividing the $p'$-part of the exact centralizer divides the displayed cyclotomic product on the right side of \eqref{eq:DevenCentralizer}, and therefore has multiplicative order at most $n$ modulo $q$.  Thus all hypotheses of Lemma~\ref{lem:two-ppd} are verified, and that lemma shows that the exact centralizer has full $\GL_n(q)$ type.

The full projective centralizer has index at most $2$ over the image of this exact centralizer, whereas \eqref{eq:DevenCentralizer} requires the factor $2\eta(1)$.  If $\eta(1)\ge2$, this already exceeds the possible projective index.  Hence $\eta(1)=1$ and the projective index is $2$.  The extra element must interchange the two $n$-dimensional totally singular eigenspaces.  If their eigenvalues are $a,b$, projective centralization gives $ca=b$ and $cb=a$, hence $b=-a$.  Therefore the conformal multiplier is $-a^2$, which is nonsquare because $q\equiv3\pmod4$.  This contradicts $t^*\in(G^*)'$.  Hence $d_1/2\notin\operatorname{cd}(S)$.

\smallskip
\noindent\emph{The nonsquare-multiplier coset when $q=3$.}
Here the only nonsquare is $-1$.  For the same hyperbolic decomposition $V=U\oplus U^*$ take
$s=\operatorname{diag}(I_U,-I_{U^*})$.  The isometry $w$ interchanging $U$ and $U^*$ satisfies $wsw^{-1}=-s$.  Since $n$ is even, $\det(w)=(-1)^n=1$, so $w\in \SO^+_{2n}(3)\le CO^+_{2n}(3)^0$.  Hence the projective centralizer has the second coset; Lemma~\ref{lem:projcent2} shows that there are no more.  Its order is therefore twice that of the exact $\GL_n(3)$ image, and the semisimple degree is
\[
 \frac12\prod_{i=1}^{n-1}(3^i+1).
\]
If half of this degree occurred in $S$, the analogue of \eqref{eq:DevenCentralizer} would require
\[
 |C_{G^*}(t^*)|_{3'}=4\eta(1)|\GL_n(3)|_{3'},
\]
while the preceding primitive-divisor argument again forces full $\GL_n(3)$ type and the projective index is at most $2$.  This is impossible.

We have treated the square-multiplier nontrivial central character and one nonsquare-multiplier central character.  The graph automorphism gives the other nonsquare-multiplier central character, completing all three double covers.
\end{proof}

\begin{theorem}\label{thm:DevenCodegree}
Let $n\ge8$ be even and let $q\equiv3\pmod4$ be an odd prime power.  If
\[
 H=\POmega^+_{2n}(q)
\]
and $G$ is a nontrivial finite group with $\operatorname{cod}(G)\subseteq\operatorname{cod}(H)$, then $G\cong H$.
\end{theorem}

\begin{proof}
Apply Proposition~\ref{prop:assemblyreduction}, using Proposition~\ref{prop:DevenB} for the elementary-abelian hypothesis and Theorem~\ref{thm:Devencover} for the prime-central-cover hypothesis.
\end{proof}

\begin{corollary}\label{cor:Dplus3mod4}
Let $n\ge7$, let $q\equiv3\pmod4$ be an odd prime power, and put
\[
 H=\POmega^+_{2n}(q).
\]
If $G$ is a nontrivial finite group and $\operatorname{cod}(G)\subseteq\operatorname{cod}(H)$, then $G\cong H$.
\end{corollary}

\begin{proof}
Use Theorem~\ref{thm:Doddcodegree} when $n$ is odd and Theorem~\ref{thm:DevenCodegree} when $n$ is even.
\end{proof}

\section{Prime covers of $\POmega_{10}^{\pm}(q)$ in odd characteristic}\label{sec:Dfive}

We treat rank five separately because the minus-type list \cite[Theorem~1.4]{Ng} contains an exceptional family.  The prime-cover character may be chosen from the usual square-multiplier part of the list, and the exceptional family is checked independently below.

\begin{theorem}\label{thm:Dfivecover}
Let $q$ be odd, let $\varepsilon\in\{+1,-1\}$, and put
\[
 S=\POmega^{\varepsilon}_{10}(q),\qquad X=\Spin^{\varepsilon}_{10}(q).
\]
Set
\[
 \eta=\begin{cases}
 +1,&q\equiv3\pmod4,\\
 -1,&q\equiv1\pmod4.
 \end{cases}
\]
Let $Y$ be the nontrivial prime central cover of $S$ arising from the cyclic center of $X$.  Then $Y$ has a faithful irreducible character $\chi$ of degree
\begin{equation}\label{eq:DfiveD}
 D_{\varepsilon,\eta}
 =\frac{(q^5-\varepsilon)(q^4+\varepsilon\eta)}{2(q-\eta)}
\end{equation}
such that
\[
 \frac{\chi(1)}2\notin\operatorname{cd}(S).
\]
In particular, the prime-central-cover condition in \cite[Theorem D]{HMII} holds for $\POmega^{\varepsilon}_{10}(q)$.
\end{theorem}

\begin{proof}
The finite dual of $X$ is
\[
 G^*=P(CO^{\varepsilon}_{10}(q)^0).
\]
The center of $X$ is cyclic of order
\[
 m=(4,q^5-\varepsilon)\in\{2,4\},
\]
and $G^*/(G^*)'$ is cyclic of the same order.  Moreover, the image of $\SO^{\varepsilon}_{10}(q)$ in this quotient is the subgroup of order $2$; equivalently,
\[
 \POmega^{\varepsilon}_{10}(q)<\PSO^{\varepsilon}_{10}(q)
\]
has index $2$.  Under the perfect pairing $Z(X)\times G^*/(G^*)'\to\C^\times$, this subgroup is the annihilator of the unique subgroup of $Z(X)$ of index $2$.  Consequently an element of $\PSO^{\varepsilon}_{10}(q)\setminus\POmega^{\varepsilon}_{10}(q)$ represents the unique order-two class in the cyclic abelianization.

Choose a nondegenerate $2$-space $W$ of type $\eta$ and let $s$ act as $-I$ on $W$ and as $I$ on $W^\perp$.  The spinor norm of $-I_W$ is the discriminant square class of $W$.  For a hyperbolic $2$-space this class is $-1$, hence is nonsquare when $q\equiv3\pmod4$; when $q\equiv1\pmod4$, an anisotropic $2$-space has nonsquare discriminant.  Thus our choice of $\eta$ makes the spinor norm of $s$ nontrivial, and
\[
 s^*\in \PSO^{\varepsilon}_{10}(q)\setminus\POmega^{\varepsilon}_{10}(q).
\]
Hence $s^*(G^*)'$ is the unique element of order $2$ in $G^*/(G^*)'$.

This is the square-multiplier case $m=n$, $k=1$ in \cite[Lemma~2.5(ii)]{Ng}.  The $(-1)$-eigenspace has type $\eta$ and the $(+1)$-eigenspace has type $\varepsilon\eta$.  The exact orthogonal centralizer is
\[
 C_{\SO^\varepsilon_{10}(q)}(s)
 =S\bigl(O^\eta_2(q)\times O^{\varepsilon\eta}_8(q)\bigr),
\]
where $S$ denotes the pairs whose determinants multiply to $1$.  Thus
\[
 |C_{\SO^\varepsilon_{10}(q)}(s)|_{p'}
 =2|\SO^\eta_2(q)|_{p'}|\SO^{\varepsilon\eta}_8(q)|_{p'}.
\]
Adjoining the scalar torus to both the ambient group and the exact centralizer and then quotienting by that same scalar torus cancels in the index.  There is no further projective-centralizing coset, since an element satisfying $gsg^{-1}=-s$ would have to interchange the $2$- and $8$-dimensional eigenspaces.  The standard orthogonal order formula therefore gives
\[
 |G^*:C_{G^*}(s^*)|_{p'}
 =\frac{(q^5-\varepsilon)(q^4+\varepsilon\eta)}{2(q-\eta)}
 =D_{\varepsilon,\eta}.
\]
The perfect pairing identifies the coset of $s^*$ with the common central character of its Lusztig series; see \cite[Lemma 4.4(i)]{NT} and \cite[Lemma 2.2]{HMII}.  Since this coset has order $2$, the resulting semisimple character factors to $Y$ and is faithful on $Z(Y)$, hence faithful on $Y$.

Put $x=q^4$ and
\[
 d_{\varepsilon,\eta}=\frac{D_{\varepsilon,\eta}}2.
\]
Suppose that $d_{\varepsilon,\eta}\in\operatorname{cd}(S)$.  Inflating to $X$ and applying \cite[Theorem~1.4]{Ng}, the generic nontrivial degrees that must be compared are
\[
 A_\varepsilon=\frac{(q^5-\varepsilon)(x+\varepsilon q)}{q^2-1},\qquad
 B=\frac{q^{10}-q^2}{q^2-1},
\]
\[
 C_\varepsilon=\frac{(q^5-\varepsilon)(x-\varepsilon)}{2(q+1)},\qquad
 D_\varepsilon=\frac{(q^5-\varepsilon)(x+\varepsilon)}{2(q-1)},
\]
and
\[
 E_{\varepsilon,+}=\frac{(q^5-\varepsilon)(x+\varepsilon)}{q-1},\qquad
 E_{\varepsilon,-}=\frac{(q^5-\varepsilon)(x-\varepsilon)}{q+1}.
\]

Assume first that $\eta=+1$, so $q\equiv3\pmod4$.  After cancelling the common factor $q^5-\varepsilon$, all nontrivial equalities reduce to
\[
 (q-3)x=\varepsilon(3q-1),\qquad
 (3q-5)x=\varepsilon(5q-3),\qquad
 q(3q-1)x=\varepsilon(4q^2-q-1).
\]
They are impossible for odd $q\ge3$, while the remaining members of the list are $2d_{\varepsilon,+}$ or $4d_{\varepsilon,+}$.

Now assume that $\eta=-1$, so $q\equiv1\pmod4$ and $q\ge5$.  Here $C_\varepsilon=2d_{\varepsilon,-}$ and $E_{\varepsilon,-}=4d_{\varepsilon,-}$.  The remaining comparisons, apart from $B$, reduce to
\[
 (q-5)x=\varepsilon(5q-1),\qquad
 (q+3)x=-\varepsilon(3q+1),\qquad
 (3q+5)x=-\varepsilon(5q+3).
\]
For $q=5$ the first is already impossible, and for $q>5$ its left side has magnitude larger than the right side; the other two are impossible by sign or size.  For the degree $B$, after clearing the positive denominator, the difference has numerator
\[
 -(q+1)(x-1)(3q^2x+4q^2+qx+q-1)
\]
when $\varepsilon=+1$, and
\[
 -(q+1)(x+1)(3q^2x-4q^2+qx-q+1)
\]
when $\varepsilon=-1$.  Both are negative for $q\ge5$.

It remains only the two exceptional minus-type degrees in \cite[Theorem~1.4]{Ng}.  When $\varepsilon=-1$ and $\eta=+1$, the inequalities already used above are
\[
 (q-1)(q^2+1)(q^3-1)(q^4+1)-d_{-,+}>0
\]
and
\[
 \frac{q^2(q^4+1)(q^5+1)}{q+1}-d_{-,+}>0.
\]
When $\varepsilon=-1$ and $\eta=-1$, direct subtraction gives
\[
 (q-1)(q^2+1)(q^3-1)(q^4+1)-d_{-,-}
 =\frac{q^4+1}{4}
 (4q^6-4q^5+3q^4-7q^3+3q^2-3q+3)>0
\]
and
\[
 \frac{q^2(q^4+1)(q^5+1)}{q+1}-d_{-,-}
 =\frac{(2q-1)(2q+1)(q^4+1)(q^4-q^3+q^2-q+1)}4>0.
\]
Thus $d_{\varepsilon,\eta}$ is absent from the complete low-degree list in every case, and hence
\[
 d_{\varepsilon,\eta}\notin\operatorname{cd}(S).
\]
\end{proof}

\section{Prime covers of $\POmega_{12}^{\pm}(q)$ in odd characteristic}\label{sec:Dsix}

We next treat rank six.  The minus group has only one nontrivial prime central cover, whereas the plus group has three double covers.  The square-multiplier cover is still controlled by the generic low-degree list \cite[Theorem~1.4]{Ng}.  The two half-spin covers require a nonsquare-multiplier parameter; for $q=3$ we use the explicit low-degree calculation in the proof of \cite[Proposition~8.3]{Ng}, and for $q>3$ we use direct centralizer comparisons.

\begin{theorem}\label{thm:Dsixcover}
Let $q$ be odd.
\begin{enumerate}
\item Put $S=\POmega^-_{12}(q)$ and $X=\Spin^-_{12}(q)$.  Then $X$ has a faithful $\chi\in\Irr(X)$ such that
\[
 \frac{\chi(1)}2\notin\operatorname{cd}(S).
\]
\item Put $S=\POmega^+_{12}(q)$ and $\widetilde S=\Spin^+_{12}(q)$.  For every subgroup $Z_0<Z(\widetilde S)$ of order $2$, the double cover
\[
 Y=\widetilde S/Z_0
\]
has a faithful $\chi\in\Irr(Y)$ such that
\[
 \frac{\chi(1)}2\notin\operatorname{cd}(S).
\]
\end{enumerate}
Thus the prime-central-cover condition in \cite[Theorem D]{HMII} holds for $\POmega_{12}^{\pm}(q)$.
\end{theorem}

\begin{proof}
Write
\[
 G^*=P(CO^{\varepsilon}_{12}(q)^0)
\]
for the finite dual of $\Spin^{\varepsilon}_{12}(q)$.  We repeatedly use the perfect pairing
\[
 Z(\Spin^{\varepsilon}_{12}(q))\times G^*/(G^*)'\longrightarrow\C^\times
\]
from \cite[Lemma 4.4(i)]{NT} and \cite[Lemma 2.2]{HMII}; the coset of the semisimple parameter is the central character of its Lusztig series.

\smallskip
\noindent\emph{Minus type.}
Here
\[
 |Z(\Spin^-_{12}(q))|=(4,q^6+1)=2,
\]
so $X=\Spin^-_{12}(q)$ is the unique nontrivial prime central cover.  Set
\[
 \eta=\begin{cases}
 +1,&q\equiv3\pmod4,\\
 -1,&q\equiv1\pmod4.
 \end{cases}
\]
Choose a nondegenerate $2$-space $W$ of type $\eta$ and let $s$ act as $-I$ on $W$ and as $I$ on $W^\perp$.  As in the proof of Theorem~\ref{thm:Dfivecover}, the spinor norm of $s$ is nonsquare, so
\[
 s^*\in \PSO^-_{12}(q)\setminus \POmega^-_{12}(q).
\]
The square-multiplier case $m=n$, $k=1$ in \cite[Lemma~2.5(ii)]{Ng} gives a semisimple character of degree
\begin{equation}\label{eq:DsixminusD}
 D_{-,\eta}:=\frac{(q^6+1)(q^5-\eta)}{2(q-\eta)}.
\end{equation}
Its central character is nontrivial, hence it is faithful on $X$.

Put $d_{-,\eta}=D_{-,\eta}/2$.  Since $d_{-,\eta}<q^{14}$, \cite[Theorem~1.4]{Ng} is exhaustive.  The generic comparison in the proof of Theorem~\ref{thm:Dfivecover}, with $x=q^5$ and $\varepsilon=-1$, shows that $d_{-,\eta}$ is different from every generic degree in that list.  The exceptional family in \cite[Theorem~1.4]{Ng} occurs only for $(n,\alpha)=(5,-1)$, so there are no further rank-six degrees.  Hence
\[
 d_{-,\eta}\notin\operatorname{cd}(S).
\]

\smallskip
\noindent\emph{Plus type: the nontrivial square-multiplier coset.}
Now $Z(\widetilde S)\cong C_2\times C_2$.  The unique nontrivial square-multiplier coset is represented by an element of
\[
 \PSO^+_{12}(q)\setminus \POmega^+_{12}(q).
\]
Choose the same type $\eta$ of $2$-space as above and let $s=-I$ on that space and $I$ on its orthogonal complement.  Again its spinor norm is nonsquare.  The case $m=n$, $k=1$ in \cite[Lemma~2.5(ii)]{Ng} gives a semisimple character of degree
\begin{equation}\label{eq:DsixplusD0}
 D_{+,\eta}:=\frac{(q^6-1)(q^5+\eta)}{2(q-\eta)},
\end{equation}
faithful on the corresponding double cover.  Put $d_{+,\eta}=D_{+,\eta}/2$.  Applying the same generic comparison from the proof of Theorem~\ref{thm:Dfivecover}, now with $x=q^5$ and $\varepsilon=+1$, gives
\[
 d_{+,\eta}\notin\operatorname{cd}(S).
\]
There are no rank-six exceptional degrees beyond that generic list.

It remains to treat the two nonsquare-multiplier central characters.  The graph automorphism exchanges them, so it is enough to construct and treat one.

\smallskip
\noindent\emph{The nonsquare-multiplier coset when $q=3$.}
Lemmas~8.1--8.2 and the proof of \cite[Proposition~8.3]{Ng} give the complete list below $4\cdot3^{15}$.  In the nonsquare-multiplier $\GL_6(3)$ case, the proof gives characters of degree
\[
 D_1=\frac12\prod_{i=1}^{5}(3^i+1)=11\,204\,480.
\]
The perfect pairing between the center and the dual abelianization identifies the parameter coset with the common central character of its Lusztig series; see \cite[Lemma~4.4(i)]{NT} and \cite[Lemma~2.2]{HMII}.  Since the parameter has nonsquare multiplier, one of these characters lies over the corresponding nonsquare central character and is faithful on the associated double cover.  Set
\[
 d_1=D_1/2=5\,602\,240<4\cdot3^{15}.
\]
We now use the explicit list in the proof of \cite[Proposition~8.3]{Ng}.  The other $\GL_6/\GU_6$ semisimple degrees occurring there are
\[
 10\,318\,880\qquad\text{and}\qquad 5\,159\,440.
\]
The $m=5$ and $m=6$, $k=1,5$ linear cases give
\[
 44\,044,\qquad 44\,408,\qquad 22\,022,
\]
and the $k=2,4$ linear cases give
\[
 27\,538\,511,\qquad 21\,493\,472.
\]
The only additional $k=1,5$ degree below the cutoff is
\[
 22\,022\cdot\frac{(3^5+1)(3^4-3)}8=52\,390\,338.
\]
None equals $d_1$.  For the unipotent characters, the three generic degrees from Proposition~3.4 have already been excluded by the preceding generic comparison, while every additional unipotent degree in Lemma~8.2 is divisible by $3$ by the symbol-degree formula preceding Proposition~3.4; $d_1$ is a $3'$-number.  Hence $d_1\notin\operatorname{cd}(S)$.  The graph automorphism gives the same conclusion for the other nonsquare central character.

\smallskip
\noindent\emph{The nonsquare-multiplier coset when $q>3$ and $q\equiv3\pmod4$.}
Choose a nonsquare $\tau\ne-1$ and, for a hyperbolic decomposition $V=U\oplus U^*$ with $\dim U=6$, put
\[
 s=\operatorname{diag}(I_U,\tau I_{U^*}).
\]
Then $s\in CO^+_{12}(q)^0$ and
\[
 C_{G^*}(s^*)\cong \GL_6(q).
\]
Thus the associated faithful character of one nonsquare double cover has degree
\begin{equation}\label{eq:DsixGL6}
 D_1=\prod_{i=1}^{5}(q^i+1).
\end{equation}
Suppose $D_1/2=\psi(1)$ for some $\psi\in\Irr(S)$ and inflate $\psi$ to $\widetilde S$.  If $((t^*),\eta)$ is its Jordan decomposition, then $t^*\in(G^*)'$ and
\begin{equation}\label{eq:DsixGLcentralizer}
 |C_{G^*}(t^*)|_{p'}=2\eta(1)|\GL_6(q)|_{p'},\qquad \eta(1)\in\{1,2,4\}.
\end{equation}
We use the following elementary consequence of Zsigmondy's theorem throughout: if $q$ is odd, $m>1$, and $r$ is a primitive prime divisor of $q^m-1$, then $r$ is odd, since $\operatorname{ord}_2(q)=1$.  Zsigmondy's only exceptional case with exponent $m>2$ is $(q,m)=(2,6)$; the exponent-$2$ power-of-two exception is irrelevant here.  In particular, for odd $q$ primitive divisors exist for each of the exponents $5,6,8$ used below.

Let $r_6,r_5$ be primitive prime divisors of $q^6-1$ and $q^5-1$.  They are therefore odd.  We check rank six directly.

If a proper orthogonal factor contains $r_5$, then its rank is at least $5$.
Thus, if $r_6$ lies in a rank-three minus factor or in any rank-four
orthogonal factor, the remaining rank is too small to supply $r_5$.  At
rank exactly $5$, none of the factors $q^{2i}-1$ with $1\le i\le4$ can
contain a prime of order $5$, so $r_5$ must occur in the terminal factor
$q^5-\varepsilon$.  Hence the sign is plus.  Indeed, $O_{10}^+(q)$ has
$p'$-part built from
\[
 (q^5-1)\prod_{i=1}^{4}(q^{2i}-1),
\]
whereas $O_{10}^-(q)$ has terminal factor $q^5+1$; thus an order-$5$ primitive divisor forces the plus sign.  But $O_{10}^+(q)$ also contains the factor $q^8-1$.  By the preceding Zsigmondy observation there is a primitive prime divisor $r_8$ of $q^8-1$; it is odd and $\operatorname{ord}_{r_8}(q)=8$.  This cannot divide the right side of \eqref{eq:DsixGLcentralizer}, since $|\GL_6(q)|_{p'}$ has only factors $q^i-1$ with $i\le6$ and $\eta(1)\in\{1,2,4\}$.  Thus neither rank-five sign is possible, and rank six orthogonal type makes the parameter central.

Now list the extension-field possibilities.  In a linear factor
$\GL_a(q^e)$ of rank cost $ea\le6$, the occurrence of $r_5$ requires
$5\mid ej$ for some $1\le j\le a$; hence $ea\ge5$.  The occurrence of $r_6$ in a second linear extension-field factor costs at least $6$ units of ambient rank, since $6\mid ej$ with $j\le a$ implies $ea\ge ej\ge6$.  Thus distinct linear factors would cost at least $5+6=11>6$.  A single linear factor containing both satisfies
\[
 5\mid ej',\qquad 6\mid ej,\qquad j,j'\le a,\qquad ea\le6.
\]
The $r_6$ condition forces $ea=6$, so $(e,a)$ is one of $(1,6),(2,3),(3,2),(6,1)$.  The condition $5\mid ej'$ for some $j'\le a$ excludes the last three.  Hence $(e,a)=(1,6)$, giving $\GL_6(q)$.  For a unitary factor
$\GU_a(q^e)$ the cyclotomic factors are $q^{ej}-(-1)^j$.  Since
$\operatorname{ord}_{r_5}(q)=5$ is odd, $r_5$ cannot divide a plus factor
$q^{ej}+1$: that would force the order of $q$ modulo $r_5$ to be even.
Hence $r_5$ must occur in a minus factor $q^{ej}-1$, so $j$ is even and
$5\mid ej$.  Therefore $10\mid ej$; since $j\le a$, this gives
$ea\ge ej\ge10$, contradicting $ea\le6$.  Thus no unitary factor can contain $r_5$,
and the exact centralizer has full $\GL_6(q)$ type.

The exact centralizer is therefore of full $\GL_6(q)$ type.  Its image in the finite dual has the $p'$-part occurring as the $\GL_6(q)$ term in \eqref{eq:DsixGLcentralizer}, and Lemma~\ref{lem:projcent2} allows the full projective centralizer to enlarge this image by at most a factor $2$.  Since \eqref{eq:DsixGLcentralizer} requires the factor $2\eta(1)$, the possibilities $\eta(1)=2,4$ would require projective indices $4,8$, respectively.  Hence $\eta(1)=1$, and the projective index is exactly $2$.  Its representative interchanges the two totally singular eigenspaces.  If their eigenvalues are $a,b$, projective centralization gives $b=-a$, so the multiplier is $-a^2$, which is nonsquare because $q\equiv3\pmod4$.  This contradicts $t^*\in(G^*)'$.  Hence $D_1/2\notin\operatorname{cd}(S)$.

\smallskip
\noindent\emph{The nonsquare-multiplier coset when $q\equiv1\pmod4$.}
Write $6=2r$ with $r=3$.  Use the nonsquare-multiplier $m=n$ case \cite[Lemma~2.5(i)]{Ng} and choose $s$ with exact centralizer
\[
 C_{O^+_{12}(q)}(s)\cong O^+_6(q^2).
\]
At the conformal level the exact centralizer is
$Z(CO^+_{12}(q)^0)O^+_6(q^2)$, with the common scalar intersection removed
when passing to the finite dual.  Lemma~\ref{lem:extfieldcoset}, applied to
the plus-type extension-field factor, supplies an $\F_q$-isometry of
determinant $1$ inducing Frobenius; it sends $s$ to $-s$.  Thus the
projective centralizer has exactly one additional coset by
Lemma~\ref{lem:projcent2}.  The scalar torus cancels in the projective
index, and the associated semisimple character, faithful on one nonsquare
double cover, has degree
\begin{equation}\label{eq:DsixO6q2}
 D_2=\frac12(q^2-1)(q^6-1)(q^{10}-1).
\end{equation}
Assume $D_2/2=\psi(1)$ for some $\psi\in\Irr(S)$ and inflate to $\widetilde S$.  Then the semisimple parameter $t^*$ has square multiplier and Jordan decomposition gives
\begin{equation}\label{eq:DsixOcentralizer}
 |C_{G^*}(t^*)|_{p'}=4\eta(1)B,
 \qquad B=(q^6-1)(q^4-1)(q^8-1),
 \qquad \eta(1)\in\{1,2,4\}.
\end{equation}
Let $\rho_8$ be a primitive prime divisor of $q^8-1$.  The square-multiplier centralizer decomposition \cite[Lemma~2.5(ii)]{Ng} shows first that a linear extension-field factor cannot contain $\rho_8$, since its total rank is at most $6$.  A unitary factor containing $\rho_8$ would have a factor $q^{ki}+1$ with $ki=4$ and $i$ odd.  Hence $k=4$, and such a factor cannot also contain a primitive divisor of $q^6-1$; the remaining rank is at most $2$ and cannot supply that divisor.  Thus $\rho_8$ lies in an orthogonal factor.  That factor has rank at least $4$.  A rank-$5$ orthogonal factor would introduce a primitive divisor of order $5$ in plus type or order $10$ in minus type, neither of which occurs on the right side of \eqref{eq:DsixOcentralizer}; rank $6$ would make the parameter central.  Hence the factor has rank exactly $4$ and is $O^-_8(q)$.

Only rank $2$ remains.  Let $\rho_4$ be a primitive prime divisor of $q^4-1$.  Then $\rho_4$ is odd and $\rho_4\mid q^2+1$.  In particular $\rho_4\nmid q^4+1$, so from $q^8-1=(q^4-1)(q^4+1)$ we obtain
\[
 v_{\rho_4}(q^8-1)=v_{\rho_4}(q^4-1).
\]
Since $B=(q^6-1)(q^4-1)(q^8-1)$ and a primitive order-$4$ prime does not divide $q^6-1$, it follows that
\[
 v_{\rho_4}(B)=2v_{\rho_4}(q^4-1).
\]
Thus the right side of \eqref{eq:DsixOcentralizer} contains two copies of the $\rho_4$-part of $q^4-1$, whereas $O^-_8(q)$ supplies only one.  The remaining rank-$2$ factor must therefore contain $\rho_4$.  We consider the rank-two possibilities directly.  The linear possibilities $\GL_2(q)$ and $\GL_1(q^2)$ do not contain a primitive divisor of $q^4-1$; neither does $\GU_2(q)$, whose relevant factors divide $(q+1)(q^2-1)$, nor a product of two rank-one $q$-tori.  Among rank-two orthogonal factors, only $O_4^-(q)$ has the terminal factor $q^2+1$; $O_4^+(q)$ has $q^2-1$.  The only extension-field torus with a $q^2+1$ factor is $\GU_1(q^2)$.  Hence the remainder is exactly
\[
 O^-_4(q)\qquad\text{or}\qquad \GU_1(q^2).
\]

If the remainder is $O^-_4(q)$, then
\[
 |O^-_8(q)|_{p'}|O^-_4(q)|_{p'}=4(q^2-1)B.
\]
The product centralizer has $p'$-part $4(q^2-1)B$.  Intersecting with
the determinant-one subgroup divides this by $2$, so
\[
 |C_{\SO}(t)|_{p'}=2(q^2-1)B.
\]
By \eqref{eq:projectiveimageorder}, the exact projective image has $p'$-part
$2(q^2-1)B$.  The full projective centralizer contains this image, and hence
\[
 |C_{G^*}(t^*)|_{p'}\ge2(q^2-1)B>16B,
\]
because $q\ge5$.  This contradicts \eqref{eq:DsixOcentralizer}, whose
right-hand side is at most $16B$.

If the remainder is $\GU_1(q^2)$, then
\[
 |O^-_8(q)|_{p'}|\GU_1(q^2)|_{p'}=2B.
\]
Here the product centralizer has $p'$-part $2B$, and its determinant-one
intersection has $p'$-part $B$, so $|C_{\SO}(t)|_{p'}=B$.  By
\eqref{eq:projectiveimageorder}, the exact projective image therefore has
$p'$-part $B$.  Lemma~\ref{lem:projcent2} permits at most one additional
component of order $2$, and consequently
\[
 |C_{G^*}(t^*)|_{p'}\le 2B,
\]
contradicting the lower bound $4B$ in \eqref{eq:DsixOcentralizer}.  Thus $D_2/2\notin\operatorname{cd}(S)$.

The graph automorphism exchanges the two nonsquare central characters, so the same degree-exclusion statement holds for the second half-spin double cover.  This completes the proof.
\end{proof}

\section{Prime covers of $\POmega_{8}^{\pm}(q)$ in odd characteristic}\label{sec:Dfour}
We finish with the rank-four prime-central-cover calculation.  The low-degree theorem \cite[Theorem~1.4]{Ng} starts at rank five, so here we use only the semisimple-centralizer description \cite[Lemma~2.5]{Ng}, The plus case also uses triality.

\begin{proposition}\label{prop:Dfourminus}
Let $q$ be odd and put $S=\POmega^-_8(q)$.  The unique nontrivial double cover of $S$ has a faithful irreducible character $\chi$ such that
\[
 \frac{\chi(1)}2\notin\operatorname{cd}(S).
\]
\end{proposition}

\begin{proof}
For odd $q$ the center of $\Spin^-_8(q)$ has order $(4,q^4+1)=2$.  Set
\[
 \delta=\begin{cases}
 +1,&q\equiv3\pmod4,\\
 -1,&q\equiv1\pmod4.
 \end{cases}
\]
Choose a nondegenerate $2$-space $W$ of type $\delta$ and let $s$ act as $-I$ on $W$ and as $I$ on $W^\perp$.  Then the spinor norm of $s$ is nonsquare: for $\delta=+1$ it is the class of $-1$, and for $\delta=-1$ it is the nonsquare discriminant class of the anisotropic plane.  Hence the image $s^*$ represents the nontrivial element of the dual abelianization.  The central-character criterion \cite[Lemma~2.2]{HMII} therefore shows that the corresponding semisimple character is faithful on the double cover.

This is the square-multiplier case $m=n$, $k=1$ in \cite[Lemma~2.5(ii)]{Ng}.  The semisimple degree is
\[
 \chi(1)=D_\delta:=\frac{(q^4+1)(q^3-\delta)}{2(q-\delta)}
 =\frac{q^4+1}{2}(q^2+\delta q+1).
\]
Both factors on the right are odd when $q$ is odd.  Thus $D_\delta$ is odd, so $D_\delta/2$ is not an integer and hence cannot be an ordinary character degree of $S$.
\end{proof}

\begin{proposition}\label{prop:Dfourplus}
Let $q$ be odd and put $S=\POmega^+_8(q)$.  For each of the three subgroups $Z_0<Z(\Spin^+_8(q))$ of order $2$, the double cover
\[
 X=\Spin^+_8(q)/Z_0
\]
has a faithful $\chi\in\Irr(X)$ such that $\chi(1)/2\notin\operatorname{cd}(S)$.
\end{proposition}

\begin{proof}
Here $Z(\Spin^+_8(q))\cong C_2\times C_2$.  Set
\[
 \delta=\begin{cases}
 +1,&q\equiv3\pmod4,\\
 -1,&q\equiv1\pmod4.
 \end{cases}
\]
Choose a nondegenerate $2$-space $W$ of type $\delta$ and let $s$ act as $-I$ on $W$ and as $I$ on $W^\perp$.  As in the proof of Theorem~\ref{thm:Dfivecover}, its spinor norm is nonsquare, so $s^*(G^*)'$ is the nontrivial square-multiplier class.  The case $m=n$, $k=1$ in \cite[Lemma~2.5(ii)]{Ng} gives the semisimple degree
\[
 D_\delta:=\frac{(q^4-1)(q^3+\delta)}{2(q-\delta)}.
\]
It is faithful on the corresponding double cover by \cite[Lemma~2.2]{HMII}.  A triality automorphism of $\Spin_8^+(q)$ permutes the three order-two subgroups of its center transitively; cf. \cite[Proposition~5.4.2(ii)]{KL}.  It therefore induces isomorphisms among the three corresponding double covers and carries a faithful character above one nontrivial central character to faithful characters of the same degree above the other two.  Thus it is enough to prove that
\[
 d:=\frac{D_\delta}{2}\notin\operatorname{cd}(S).
\]
Suppose first that $q>3$ and that $d=\psi(1)$ for some $\psi\in\Irr(S)$.  Inflate $\psi$ to $\Spin^+_8(q)$ and let $(t^*,\eta)$ be its Jordan parameter.  Triviality on the center gives $t^*\in(G^*)'$ by \cite[Lemma~2.2]{HMII}; in particular a lift $t$ has square conformal multiplier.  Since $d$ is a $p'$-number, Lemma~\ref{lem:Deveneta} gives $\eta(1)\in\{1,2,4\}$.

If $q\equiv1\pmod4$, put
\[
 B_-=(q+1)(q^2-1)(q^3+1)(q^4-1).
\]
The direct semisimple-degree quotient gives $|C_{G^*}(s^*)|_{p'}=2B_-$, so a hypothetical character of degree $d$ would require
\[
 |C_{G^*}(t^*)|_{p'}=4\eta(1)B_-.
\]
Let $r_4$ and $r_6$ be primitive prime divisors of $q^4-1$ and $q^6-1$.  Both must occur in the exact square-multiplier centralizer.  A linear extension-field factor containing $r_6$ has rank cost at least $6$, so none occurs in rank $4$.  For unitary factors, a factor carrying $r_6$ has rank cost at least $3$; if it has rank cost $3$, it does not contain $r_4$ and the residual rank one cannot supply $r_4$.  Thus the only full unitary possibility is $\GU_4(q)$, whose exact projective $p'$-part is at most $B_-$ and whose full projective centralizer has $p'$-part at most $2B_-$.  For orthogonal factors, the two primitive divisors force an $O^-_6(q)$ factor; the remaining rank-one factor is one of $O_2^\pm(q)$, $\GL_1(q)$ or $\GU_1(q)$.  The exact projective image has $p'$-part at most $2B_-$ by \eqref{eq:projectiveimageorder}.  In the potentially sharp case $O^-_6(q)\times O^-_2(q)$, there is no additional projective coset: an element in such a coset would conjugate a lift $t$ to $-t$ and hence interchange the $+1$- and $-1$-primary spaces, but these have dimensions $6$ and $2$.  The same unequal-dimension argument excludes an extra coset for the other $O^-_6(q)\times O_2^\pm(q)$ possibility, while the two toral rank-one possibilities have smaller exact $p'$-part and remain below $2B_-$ even after the possible factor $2$ from Lemma~\ref{lem:projcent2}.  Thus the full projective $p'$-part is at most $2B_-$ in every case.  The only rank-four orthogonal factor compatible with the ambient form is $O^+_8(q)$, which makes the parameter central.  Hence no proper centralizer reaches $4B_-$, a contradiction.  Thus $d\notin\operatorname{cd}(S)$ when $q\equiv1\pmod4$.

We may therefore assume $q\equiv3\pmod4$.  Then $\delta=+1$, write $D_+=D_\delta$, and put
\[
 B=(q-1)(q^2-1)(q^3-1)(q^4-1).
\]
Here the finite dual is $G^*=P(CO_8^+(q)^0)$.  Its $p'$-part is
\[
 |G^*|_{p'}=(q^2-1)(q^4-1)^2(q^6-1).
\]
For the chosen semisimple parameter the Jordan correspondent is trivial,
so the semisimple degree formula gives
\[
 D_+=|G^*:C_{G^*}(s^*)|_{p'}.
\]
Consequently, without making any intermediate scalar or determinant-one
normalization,
\begin{align*}
 |C_{G^*}(s^*)|_{p'}
 &=\frac{(q^2-1)(q^4-1)^2(q^6-1)}
 {\dfrac{(q^4-1)(q^3+1)}{2(q-1)}}\\
 &=2(q-1)(q^2-1)(q^3-1)(q^4-1)=2B.
\end{align*}
This direct quotient calculation is the normalization used below.
Comparing the two Jordan degree formulas therefore yields
\begin{equation}\label{eq:Dfourpluscentralizer}
 |C_{G^*}(t^*)|_{p'}=4\eta(1)B.
\end{equation}
Let $r_3$ and $r_4$ be primitive prime divisors of $q^3-1$ and $q^4-1$.  Zsigmondy's theorem supplies both for odd $q$, and both are odd.  Equation~\eqref{eq:Dfourpluscentralizer} forces them into the exact square-multiplier centralizer.

We spell out the rank-four consequence of the decomposition \cite[Lemma~2.5(ii)]{Ng}.  Since $\operatorname{ord}_{r_i}(q)=i$, consider
first a linear extension-field factor $\GL_a(q^e)$.  If it contains $r_4$,
then $4\mid ej$ for some $1\le j\le a$, with rank cost $ae\le4$.  The only
possibilities capable of realizing exponent $4$ are
\[
 (e,j)=(1,4),(2,2),(4,1),
\]
corresponding respectively to a full $\GL_4(q)$ factor, a $\GL_2(q^2)$
factor, or a $\GL_1(q^4)$ factor (enlarging $a$ only increases the rank
cost).  The latter two do not contain a primitive divisor of $q^3-1$:
their cyclotomic exponents divide $2,4$ and $4$, respectively.  Hence a
single linear factor containing both $r_3$ and $r_4$ must be $\GL_4(q)$.
If the two primes lie in distinct linear factors, a factor carrying $r_3$
has ambient rank cost at least $3$, while a factor carrying $r_4$ has
ambient rank cost at least $4$: indeed $4\mid ej$ with $j\le a$ implies
$ea\ge ej\ge4$.  Thus the two factors would have total rank cost at least
$3+4=7>4$.

For a unitary factor $\GU_a(q^e)$ the factors are
$q^{ej}-(-1)^j$.  An odd-order primitive divisor of $q^3-1$ cannot divide
a plus factor $q^m+1$: otherwise the order of $q$ modulo that prime would
divide both $3$ and $2m$ without dividing $m$.  Thus $r_3$ must occur in a
minus factor, forcing an even unitary index $j$ with $3\mid ej$.  Since
$j\le a$, we have $6\mid ej$ and hence the ambient rank cost satisfies
$ea\ge ej\ge6$.  Therefore no unitary factor of total rank at most $4$
contains $r_3$.  Mixed
linear--unitary products are therefore excluded as well.  Thus the only
linear/unitary possibility carrying both primitive divisors is the full
$\GL_4(q)$ factor.

For orthogonal factors, the standard order formula
\[
 |O_{2a}^{\epsilon}(q)|_{p'}
   =2(q^a-\epsilon)\prod_{j=1}^{a-1}(q^{2j}-1)
\]
shows that rank at most $2$ cannot contain $r_3$.  In rank $3$, the minus
factor has distinguished term $q^3+1$, so it does not contain the primitive
order-$3$ prime $r_3$, whereas
\[
 |O_6^+(q)|_{p'}=2(q^3-1)(q^2-1)(q^4-1)
\]
contains both $r_3$ and $r_4$.  It leaves exactly rank one.  the primary decomposition \cite[Lemma~2.5]{Ng} permits either a $\pm1$-primary orthogonal remainder $O_2^\delta(q)$, $\delta=\pm$, or a non-$\pm1$ rank-one torus $\GL_1(q)$ or $\GU_1(q)$.  In rank $4$ the orthogonal factor is full, and compatibility with the ambient plus form leaves $O_8^+(q)$.  Hence the complete list is
\[
 \GL_4(q),\quad O_8^+(q),\quad O_6^+(q)\times O_2^\delta(q),\quad
 O_6^+(q)\times \GL_1(q),\quad O_6^+(q)\times \GU_1(q).
\]
There are no additional mixed cases: a linear or unitary factor carrying $r_4$ already has rank cost $4$, while an orthogonal factor carrying both primitive divisors has rank $3$, and the preceding list exhausts the rank-one remainder.

The full $\GL_4(q)$ possibility is excluded by order comparison.  The determinant-one centralizer has $p'$-part $B$, and its exact image in $P(CO^+_8(q)^0)$ has the same $p'$-part $B$ by \eqref{eq:projectiveimageorder}.  Lemma~\ref{lem:projcent2} therefore gives a full projective centralizer of $p'$-part at most $2B$, contrary to \eqref{eq:Dfourpluscentralizer}.  A full $O^+_8(q)$ factor makes $t^*$ central.  In that case the Jordan index is $1$, so the character degree is a unipotent $p'$-degree of the connected type-$D_4$ group.  By \cite[Theorem~6.8]{Mal}, the trivial character is the only such character in odd defining characteristic, hence the degree is $1$, again not $d$.

It remains to consider $O^+_6(q)\times O_2^\delta(q)$.  For $\delta=-$, the determinant-one centralizer has $p'$-part
\[
 2B\frac{q+1}{q-1}.
\]
By \eqref{eq:projectiveimageorder}, its exact image in the finite dual has $p'$-part
\[
 2B\frac{q+1}{q-1}.
\]
The full projective centralizer has this order or twice this order.  Equality with \eqref{eq:Dfourpluscentralizer} would therefore give
\[
 e(q+1)=2\eta(1)(q-1),\qquad e\in\{1,2\}.
\]
For $q>3$ and $\eta(1)\in\{1,2,4\}$ this has no solution.

For $\delta=+$ the determinant-one centralizer has $p'$-part $2B$, so its exact projective image has $p'$-part $2B$ by \eqref{eq:projectiveimageorder}.  A projective centralizer element outside the exact centralizer would satisfy $gtg^{-1}=-t$.  After normalizing the multiplier to $1$, the two orthogonal eigenspaces have dimensions $6$ and $2$, so such an element would have to interchange subspaces of unequal dimension.  Hence the projective component is trivial and
\[
 |C_{G^*}(t^*)|_{p'}=2B,
\]
again contradicting \eqref{eq:Dfourpluscentralizer}.

The two remaining toral cases are smaller still; we give the normalization explicitly.  For $O_6^+(q)\times \GL_1(q)$ the product before imposing determinant one has $p'$-part $2B$: the $O_6^+(q)$ factor contributes $2(q^3-1)(q^2-1)(q^4-1)$ and the split torus contributes $q-1$.  Here the $\GL_1(q)$ primary factor acts on the paired primary space $U\oplus U^*$ as $a\mapsto\operatorname{diag}(a,a^{-1})$, so its determinant on the ambient orthogonal space is $1$.  Thus the determinant map on the product is exactly the determinant-sign map on $O_6^+(q)$, which is onto $\{\pm1\}$; intersection with the determinant-one subgroup therefore divides the product order by exactly $2$ and gives $B$.  By \eqref{eq:projectiveimageorder}, the exact projective image has $p'$-part $B$.  Lemma~\ref{lem:projcent2} gives full projective centralizer at most $2B$.

For $O_6^+(q)\times \GU_1(q)$ the unitary rank-one primary block is likewise paired with its dual; equivalently a norm-one scalar has determinant $1$ on the underlying two-dimensional $\F_q$-space.  Hence it contributes no determinant sign.  Replacing the split torus factor $q-1$ by $q+1$, the product has $p'$-part $2B(q+1)/(q-1)$, its determinant-one intersection has $p'$-part $B(q+1)/(q-1)$, and by \eqref{eq:projectiveimageorder} its exact projective image has the same $p'$-part
\[
 B\frac{q+1}{q-1}.
\]
Again Lemma~\ref{lem:projcent2} allows at most one further factor $2$, so the full projective centralizer has $p'$-part at most $2B(q+1)/(q-1)<4B$ for $q>3$.  Both toral cases are incompatible with \eqref{eq:Dfourpluscentralizer}, whose right side is at least $4B$.

It remains only $q=3$.  Here $D_+=560$ and $d=280$.  The GAP Character Table Library table 
\texttt{CharacterTable("O8+(3)")} is the table of the simple group $\POmega^+_8(3)$: the CTblLib information page lists the construction \texttt{POmega(1,8,3)} and gives its order $4952179814400$.  Its origin is the ATLAS.  This table has no irreducible character of degree $280$; equivalently, the command
\[
 \texttt{280 in Set(List(Irr(CharacterTable("O8+(3)")),Degree));}
\]
returns \texttt{false}; see \cite{CTblLibO8p3}.  Thus $280\notin\operatorname{cd}(\POmega^+_8(3))$.  Triality now gives the same conclusion for all three double covers.
\end{proof}

\begin{corollary}\label{cor:Dfourcovers}
For every odd prime power $q$, the prime-central-cover condition required in the codegree reduction holds for both $\POmega^-_8(q)$ and $\POmega^+_8(q)$.
\end{corollary}

\begin{theorem}\label{thm:Dlowcodegree}
Let $q$ be odd, let $4\le n\le6$, let $\varepsilon\in\{+,-\}$, and suppose that
\[
 S=\POmega_{2n}^{\varepsilon}(q)
\]
is simple.  If $G$ is a nontrivial finite group satisfying
\[
 \cod(G)\subseteq\cod(S),
\]
then $G\cong S$.
\end{theorem}

\begin{proof}
Apply Proposition~\ref{prop:assemblyreduction}.  For its elementary-abelian hypothesis use \cite[Proposition 4.3]{HMII} when $n=4$ and Proposition~\ref{prop:DallB} when $n=5,6$; for its prime-central-cover hypothesis use Corollary~\ref{cor:Dfourcovers}, Theorem~\ref{thm:Dfivecover}, and Theorem~\ref{thm:Dsixcover}, respectively.
\end{proof}

\section{Plus-type even-dimensional orthogonal groups of odd rank with $q\equiv1\pmod4$}\label{sec:Dplusone}
Throughout this section
\[
 S=\POmega^+_{2n}(q),\qquad n\ge7\text{ odd},\qquad q\equiv1\pmod4.
\]
Here $(4,q^n-1)=4$, so the generic Schur multiplier is cyclic of order $4$ and there is a unique prime central cover
\[
 Y=\Omega^+_{2n}(q)=\Spin^+_{2n}(q)/Z_2,
\]
where $Z_2$ is the unique subgroup of order $2$ in the cyclic center of the spin group.

\begin{proposition}\label{prop:DplusoneB}
Conjecture B holds for $S=\POmega^+_{2n}(q)$ when $n\ge7$ is odd and $q\equiv1\pmod4$.
\end{proposition}
\begin{proof}
This is the corresponding specialization of Proposition~\ref{prop:DallB}, proved uniformly for both signs, every odd $q$, and every $n\ge5$.  We state it separately here so that the prime-cover argument below can be read without referring back to the uniform formulation.
\end{proof}

\begin{theorem}\label{thm:Dplusonecover}
There is a faithful $\chi\in\Irr(Y)$ of degree
\[
 \chi(1)=\frac{(q^n-1)(q^{n-1}+1)}{2(q-1)}
\]
such that $\chi(1)/2\notin\operatorname{cd}(S)$. Hence the prime-central-cover condition of \cite[Theorem~D]{HMII} holds for $S$.
\end{theorem}
\begin{proof}
Put $X=\Spin^+_{2n}(q)$ and $G^*=P(CO^+_{2n}(q)^0)$. Choose a hyperbolic plane $W=\langle e,f\rangle$ and a nonsquare $c\in\mathbb F_q^\times$. Let $s$ act by $e\mapsto ce$, $f\mapsto c^{-1}f$ on $W$ and trivially on $W^\perp$. Then $s\in \SO^+_{2n}(q)$ has multiplier $1$ and nontrivial spinor norm (the square class of $c$), so
\[
 s^*\in \PSO^+_{2n}(q)\setminus \POmega^+_{2n}(q).
\]
the semisimple-centralizer calculation \cite[Lemma~2.5 and Theorem~1.4]{Ng} gives, up to the scalar conformal factor, centralizer type $O_2^+(q)\times O^+_{2n-2}(q)$, and the associated semisimple character has degree
\begin{equation}\label{eq:DplusoneD}
 D=\frac{(q^n-1)(q^{n-1}+1)}{2(q-1)}.
\end{equation}
This is exactly the fifth degree displayed in \cite[Theorem 1.4]{Ng} after substituting $\alpha=+1$.

The standard perfect pairing $Z(X)\times G^*/(G^*)'\to\mathbb C^\times$ identifies the coset of a semisimple parameter with the common central character of its Lusztig series; cf. \cite[Lemma 4.4(i)]{NT} and \cite[Lemma 2.2]{HMII}. Here $Z(X)\cong C_4$, and
\[
 \POmega^+_{2n}(q)<\PSO^+_{2n}(q)<G^*
\]
has successive indices $2$. Thus $s^*$ represents the unique element of order $2$ in $G^*/(G^*)'$. Its central character has order $2$ and kernel $Z_2$. Hence the character of degree $D$ descends to a faithful $\chi\in\Irr(Y)$.

Set $x=q^{n-1}$ and
\[
 d=D/2=\frac{(q^n-1)(x+1)}{4(q-1)}.
\]
Since $n\ge7$, $d<q^{4n-10}$. If $d\in\operatorname{cd}(S)$, inflation gives a character of $X$ of degree $d$. Theorem~1.4 of \cite{Ng} is exhaustive below this bound. For $\alpha=+1$ the nontrivial possibilities are
\[
 A=\frac{(q^n-1)(x+q)}{q^2-1},\quad B=\frac{q^{2n}-q^2}{q^2-1},\quad C=\frac{(q^n-1)(x-1)}{2(q+1)},\quad D=\frac{(q^n-1)(x+1)}{2(q-1)},
\]
and
\[
 E_+=\frac{(q^n-1)(x+1)}{q-1},\qquad E_-=\frac{(q^n-1)(x-1)}{q+1}.
\]
Clearly $d=D/2$, so $d\ne D,E_+$. Equality $d=A$ gives $(q+1)(x+1)=4(x+q)$, hence $(q-3)x=3q-1$, impossible since $q\ge5$. Equality $d=C$ gives $(q+1)(x+1)=2(q-1)(x-1)$ and the same impossible equation. Equality $d=E_-$ gives $(q+1)(x+1)=4(q-1)(x-1)$, hence $(3q-5)x=5q-3$, impossible.

Finally, $d=B$ gives, using $q^n=qx$ and cancelling $x+1$ and $q-1$,
\[
 (q+1)(qx-1)=4q^2(x-1),
\]
so $q(3q-1)x=4q^2-q-1$, impossible because $x=q^{n-1}\ge q^6$. Thus $d$ is absent from the exhaustive list \cite[Theorem~1.4]{Ng}, and $\chi(1)/2\notin\operatorname{cd}(S)$.
\end{proof}

\begin{theorem}\label{thm:Dplusonecodegree}
If $G$ is a nontrivial finite group satisfying $\operatorname{cod}(G)\subseteq\operatorname{cod}(S)$, then $G\cong S$.
\end{theorem}
\begin{proof}
Apply Proposition~\ref{prop:assemblyreduction}, using Proposition~\ref{prop:DplusoneB} for the elementary-abelian hypothesis and Theorem~\ref{thm:Dplusonecover} for the prime-central-cover hypothesis.
\end{proof}

\section{Minus-type even-dimensional orthogonal groups of odd rank with $q\equiv3\pmod4$}\label{sec:Dminusodd}
We next treat the closest minus-type analogue.  Throughout this section
\[
 S=\POmega^-_{2n}(q),\qquad n\ge7\text{ odd},\qquad q=p^f\equiv3\pmod4.
\]
Here the generic Schur multiplier is cyclic of order
\[
 (4,q^n+1)=4.
\]
Consequently there is a unique prime central cover of $S$, namely
\[
 Y=\Omega^-_{2n}(q)=\Spin^-_{2n}(q)/Z_2,
\]
where $Z_2$ is the unique subgroup of order $2$ in the cyclic center of the spin group.

\begin{proposition}\label{prop:DminusB}
Conjecture B holds for $S=\POmega^-_{2n}(q)$ when $n\ge7$ is odd and $q\equiv3\pmod4$.
\end{proposition}

\begin{proof}
This is the corresponding specialization of Proposition~\ref{prop:DallB}, proved uniformly for both signs, every odd $q$, and every $n\ge5$.  We state it separately here so that the prime-cover argument below can be read without referring back to the uniform formulation.
\end{proof}

\begin{theorem}\label{thm:Dminuscover}
Let $n\ge7$ be odd and $q\equiv3\pmod4$, and put
\[
 S=\POmega^-_{2n}(q),\qquad Y=\Omega^-_{2n}(q).
\]
Then there is a faithful character $\chi\in\Irr(Y)$ of degree
\[
 \chi(1)=\frac{(q^n+1)(q^{n-1}-1)}{2(q-1)}
\]
such that
\[
 \frac{\chi(1)}2\notin\operatorname{cd}(S).
\]
Hence the conclusion of Theorem~D of \cite{HMII} holds for $S$.
\end{theorem}

\begin{proof}
Let
\[
 X=\Spin^-_{2n}(q),\qquad
 G^*=P(CO^-_{2n}(q)^0)
\]
be the simply connected group and its finite dual.  The discussion in \cite[Remark~7.3 and Lemma~7.4]{TZ} shows that the determinant--multiplier description of $CO_{2n}(q)^0$ used in plus type remains valid in minus type; see \cite[Remark 7.3, Lemma 7.4 and the paragraph following Lemma 7.4]{TZ}.

Choose a hyperbolic $2$-space $W$ in the natural minus-type space and let a lift $s$ act as $-I$ on $W$ and as $I$ on $W^\perp$.  Then $s\in \SO^-_{2n}(q)$, its conformal multiplier is $1$, and its spinor norm is the square class of $-1$, hence is nontrivial because $q\equiv3\pmod4$.  Thus the image $s^*\in G^*$ lies in $\PSO^-_{2n}(q)\setminus \POmega^-_{2n}(q)$.  the case $m=n$, $k=1$, $\beta=+$ in \cite[Lemma~2.5(ii)]{Ng} gives its exact centralizer (up to the scalar conformal factor) of type
\[
 O_2^+(q)\times O^-_{2n-2}(q).
\]
The corresponding semisimple character has degree
\begin{equation}\label{eq:DminusD}
 D=\frac{(q^n+1)(q^{n-1}-1)}{2(q-1)}.
\end{equation}
This is one of the two characters of this degree in the complete low-degree list \cite[Theorem~1.4]{Ng} for $\Spin^-_{2n}(q)$; substitute $\alpha=-1$ in \cite[Theorem 1.4]{Ng}.

We next identify its central kernel.  The standard duality between the center of the simply connected group and the abelianization of the finite dual gives a perfect pairing
\[
 Z(X)\times G^*/(G^*)'\longrightarrow\C^\times,
\]
and the class of a semisimple parameter determines the common central character of its Lusztig series; cf. \cite[Lemma 4.4(i)]{NT} and \cite[Lemma~2.2]{HMII}.  In our case both groups are cyclic of order $4$.  The subgroup
\[
 \POmega^-_{2n}(q)<\PSO^-_{2n}(q)<G^*
\]
has successive indices $2$.  Since $s^*$ lies in $\PSO^-_{2n}(q)\setminus \POmega^-_{2n}(q)$, its class in $G^*/(G^*)'$ is the unique element of order $2$.  Hence the common central character of $\mathcal E(X,s^*)$ has order $2$ and therefore has kernel equal to the unique subgroup $Z_2$ of order $2$ in $Z(X)\cong C_4$.  The semisimple character of degree $D$ consequently descends to a faithful character $\chi\in\Irr(X/Z_2)=\Irr(Y)$.

It remains to show that $D/2$ is not a character degree of $S$.  Put $x=q^{n-1}$ and
\[
 d=\frac D2=\frac{(q^n+1)(x-1)}{4(q-1)}.
\]
Since $n\ge7$, we have $d<q^{4n-10}$, so Theorem~1.4 of \cite{Ng} applies to every character of $X$ of degree $d$.  For $\alpha=-1$, the possible nontrivial degrees below this bound are
\[
 A=\frac{(q^n+1)(x-q)}{q^2-1},\qquad
 B=\frac{q^{2n}-q^2}{q^2-1},
\]
\[
 C=\frac{(q^n+1)(x+1)}{2(q+1)},\qquad
 D=\frac{(q^n+1)(x-1)}{2(q-1)},
\]
and
\[
 E_+=\frac{(q^n+1)(x-1)}{q+1},\qquad
 E_-=\frac{(q^n+1)(x+1)}{q-1}.
\]
We compare $d$ with these six numbers.  Clearly $d=D/2$, so $d\ne D$.  Equality $d=A$ gives
\[
 (q+1)(x-1)=4(x-q),
\]
that is $(q-3)x=1-3q$, impossible.  Equality $d=C$ gives
\[
 (q+1)(x-1)=2(q-1)(x+1),
\]
that is $(q-3)x=3q-1$; this is impossible for $q=3$, and for $q>3$ because $x=q^{n-1}$.  Equality $d=E_+$ would force $q+1=4(q-1)$, while equality $d=E_-$ gives $x-1=4(x+1)$, which is impossible.

Finally, if $d=B$, then using $q^n=qx$ and cancelling $x-1$ and $q-1$ yields
\[
 (q+1)(qx+1)=4q^2(x+1),
\]
so
\[
 q(3q-1)x=q+1-4q^2,
\]
which is impossible.  Therefore $d$ is not the degree of any irreducible character of $X$, and a fortiori
\[
 d\notin\operatorname{cd}(S).
\]
This proves the theorem.
\end{proof}

\begin{theorem}\label{thm:Dminuscodegree}
Let $n\ge7$ be odd and $q\equiv3\pmod4$, and set
\[
 S=\POmega^-_{2n}(q).
\]
If $G$ is a nontrivial finite group satisfying
\[
 \operatorname{cod}(G)\subseteq\operatorname{cod}(S),
\]
then $G\cong S$.
\end{theorem}

\begin{proof}
Apply Proposition~\ref{prop:assemblyreduction}, using Proposition~\ref{prop:DminusB} for the elementary-abelian hypothesis and Theorem~\ref{thm:Dminuscover} for the prime-central-cover hypothesis.
\end{proof}

\section{Minus-type even-dimensional orthogonal groups of even rank with $q\equiv3\pmod4$}\label{sec:Dminuseven}
We now treat the remaining even-rank minus-type block.  Throughout this section
\[
 S=\POmega^-_{2n}(q),\qquad n\ge8\text{ even},\qquad q=p^f\text{ odd}.
\]
Since $q^n\equiv1\pmod4$,
\[
 (4,q^n+1)=2,
\]
so $X=\Spin^-_{2n}(q)$ is the unique nontrivial prime central cover of $S$.

\begin{proposition}\label{prop:DminusevenB}
Conjecture B holds for $S=\POmega^-_{2n}(q)$ when $n\ge8$ is even and $q$ is odd.
\end{proposition}

\begin{proof}
This is the corresponding specialization of Proposition~\ref{prop:DallB}, proved uniformly for both signs, every odd $q$, and every $n\ge5$.  We state it separately here so that the prime-cover argument below can be read without referring back to the uniform formulation.
\end{proof}

\begin{theorem}\label{thm:Dminusevencover}
Let $n\ge8$ be even and $q$ odd.  Put
\[
 S=\POmega^-_{2n}(q),\qquad X=\Spin^-_{2n}(q).
\]
There is a faithful $\chi\in\Irr(X)$ of degree
\[
 \chi(1)=\prod_{\substack{1\le j<n\\ j\ \mathrm{odd}}}(q^{2j}-1)
\]
such that $\chi(1)/2\notin\operatorname{cd}(S)$.  Hence Theorem~D of \cite{HMII} holds for this family.
\end{theorem}

\begin{proof}
Let
\[
 G^*=P(CO^-_{2n}(q)^0)
\]
be the finite dual of $X$.  Choose a nonsquare $\tau\in\F_q^\times$ and $\lambda\in\F_{q^2}$ with $\lambda^2=\tau$; then $\lambda^q=-\lambda$.  On an $n$-dimensional minus-type orthogonal space over $\F_{q^2}$, viewed by restriction of scalars as the natural $2n$-dimensional minus-type space over $\F_q$, let $s$ act as multiplication by $\lambda$.  Then $s\in CO^-_{2n}(q)^0$, $\tau(s)=\tau$ is nonsquare, and its characteristic polynomial is $(x^2-\tau)^n$.  the nonsquare-multiplier case \cite[Lemma~2.5(i)]{Ng} gives
\[
 C_{O^-_{2n}(q)}(s)\cong O^-_n(q^2).
\]
Every element of this $O^-_n(q^2)$ has $\F_q$-determinant equal to the norm of its $\F_{q^2}$-determinant $\pm1$, hence determinant $1$.  Therefore
\[
 C_{CO^-_{2n}(q)^0}(s)=O^-_n(q^2)\cdot Z_{q-1}.
\]

The field-automorphism coset does \emph{not} lie in the finite dual in this
minus-type case.  Put
\[
 \delta(g)=\det(g)\tau(g)^{-n}\in\{\pm1\},
\]
where $\tau(g)$ is the multiplier.  The kernel of $\delta$ is
$CO^-_{2n}(q)^0$.  By Lemma~\ref{lem:extfieldcoset}, a semilinear
representative inducing the nontrivial $\F_{q^2}/\F_q$ automorphism has
multiplier $1$ and $\F_q$-determinant $-1$, so it has $\delta=-1$.
Every element of the exact extension-field centralizer lies in
$CO^-_{2n}(q)^0$ and hence has $\delta=1$.  Therefore the whole semilinear
coset has $\delta=-1$, and no element of $CO^-_{2n}(q)^0$ conjugates $s$ to
$-s$.  Lemma~\ref{lem:projcent2} therefore shows that the projective
centralizer is exactly the image of the exact centralizer.

Writing $n=2r$, the $p'$-parts are
\[
 |\SO^-_{2n}(q)|_{p'}=(q^n+1)\prod_{i=1}^{n-1}(q^{2i}-1)
\]
and
\[
 |O^-_n(q^2)|_{p'}=2(q^n+1)\prod_{i=1}^{r-1}(q^{4i}-1).
\]
After adjoining scalars and projectivizing, the scalar factors cancel and
there is no extra projective component.  The semisimple character in
$\mathcal E(X,s^*)$ therefore has degree
\begin{equation}\label{eq:DminusevenD}
 D=\prod_{\substack{1\le j<n\\j\ \mathrm{odd}}}(q^{2j}-1).
\end{equation}
Since $\tau(s)$ is nonsquare, $s^*\notin(G^*)'$.  As $|Z(X)|=2$, \cite[Lemma~2.2]{HMII} shows that this semisimple character is faithful.

Suppose, for a contradiction, that
\[
 d:=D/2\in\operatorname{cd}(S).
\]
Inflate a character of degree $d$ to $X$, and let $(t^*,\eta)$ be its Jordan decomposition.  Triviality on $Z(X)$ gives $t^*\in(G^*)'$ by \cite[Lemma 2.2]{HMII}; in particular $t$ has square conformal multiplier.  Since $d$ is a $p'$-number, $p\nmid\eta(1)$, and Lemma~\ref{lem:Deveneta} gives $\eta(1)\in\{1,2,4\}$.  Comparing degrees with \eqref{eq:DminusevenD} yields
\begin{equation}\label{eq:Dminusevencent}
 |C_{G^*}(t^*)|_{p'}=2\eta(1)(q^n+1)\prod_{i=1}^{n/2-1}(q^{4i}-1).
\end{equation}

Choose primitive prime divisors
\[
 r_{2n}\mid q^{2n}-1,\qquad r_{2n-4}\mid q^{2n-4}-1.
\]
They exist by Zsigmondy's theorem and are odd; moreover $r_{2n}\mid q^n+1$.  Equation \eqref{eq:Dminusevencent} forces both primes into the exact square-multiplier semisimple centralizer, since the projective component and $\eta(1)$ contribute only powers of $2$.

Use the square-multiplier decomposition \cite[Lemma~2.5(ii)]{Ng}.  An orthogonal factor supplying $r_{2n}$ must have full rank $n$ and minus type; then the whole exact centralizer is the full orthogonal group, so $t^*$ is central and the corresponding character degree is $1$, contrary to $d>1$.  Hence $r_{2n}$ occurs in an extension-field unitary factor $\GU_a(q^k)$.  Its factors are $q^{ki}-(-1)^i$.  Since $ki\le ka\le n$, a primitive divisor of order $2n$ can occur only through
\[
 q^{ki}+1,\qquad ki=n.
\]
Consequently $i=a$ is odd and $ka=n$; this one factor consumes all available rank.

It must therefore also contain $r_{2n-4}$.  The alternative $r_{2n-4}\mid q^{kj}-1$ is impossible because $kj\le n<2n-4$.  Thus $r_{2n-4}\mid q^{kj}+1$.  Since $\operatorname{ord}_{r_{2n-4}}(q)=2n-4$, this gives
\[
 kj\equiv n-2\pmod{2n-4}.
\]
But $0<kj\le n<2n-4$ for $n\ge8$, and hence
\[
 kj=n-2.
\]
Therefore $k\mid n$ and $k\mid n-2$, so $k\mid2$.  Since $a=n/k$ is odd, $k\ne1$, and hence $k=2$, $a=n/2$ is odd, and $j=(n-2)/2=a-1$ is even.  But an even index $j$ in $\GU_a(q^2)$ contributes the factor $q^{2j}-1$, not $q^{2j}+1$, so it cannot contain a primitive divisor of order $2n-4=4j$.  This contradiction proves $d\notin\operatorname{cd}(S)$.
\end{proof}

\begin{theorem}\label{thm:Dminusevencodegree}
Let $n\ge8$ be even, let $q$ be odd, and set $S=\POmega^-_{2n}(q)$.  If a nontrivial finite group $G$ satisfies
\[
 \operatorname{cod}(G)\subseteq\operatorname{cod}(S),
\]
then $G\cong S$.
\end{theorem}

\begin{proof}
Apply Proposition~\ref{prop:assemblyreduction}, using Proposition~\ref{prop:DminusevenB} for the elementary-abelian hypothesis and Theorem~\ref{thm:Dminusevencover} for the prime-central-cover hypothesis.
\end{proof}

\section{Minus-type even-dimensional orthogonal groups of odd rank with $q\equiv1\pmod4$}\label{sec:Dminusoddone}
We complete the large-rank minus-type family.  Throughout this section
\[
 S=\POmega^-_{2n}(q),\qquad n\ge7\text{ odd},\qquad q=p^f\equiv1\pmod4.
\]
Since $q^n\equiv1\pmod4$, we have $(4,q^n+1)=2$.  Thus
\[
 X=\Spin^-_{2n}(q)
\]
is the unique nontrivial prime central cover of $S$.

\begin{proposition}\label{prop:DminusoddoneB}
Conjecture B holds for $S=\POmega^-_{2n}(q)$ when $n\ge7$ is odd and $q\equiv1\pmod4$.
\end{proposition}

\begin{proof}
This is the corresponding specialization of Proposition~\ref{prop:DallB}, proved uniformly for both signs, every odd $q$, and every $n\ge5$.  We state it separately here so that the prime-cover argument below can be read without referring back to the uniform formulation.
\end{proof}

\begin{theorem}\label{thm:Dminusoddonecover}
Let $n\ge7$ be odd and $q\equiv1\pmod4$, and put
\[
 S=\POmega^-_{2n}(q),\qquad X=\Spin^-_{2n}(q).
\]
There exists a faithful semisimple character $\chi\in\Irr(X)$ such that
\[
 \frac{\chi(1)}2\notin\operatorname{cd}(S).
\]
Hence Theorem~D of \cite{HMII} holds for this family.
\end{theorem}

\begin{proof}
Let
\[
 G^*=P(CO^-_{2n}(q)^0)
\]
be the finite dual of $X$.  Choose a nonsquare $\tau\in\F_q^\times$ and $\lambda\in\F_{q^2}$ with $\lambda^2=\tau$.  Put $m=n-1$, which is even.  On a $2m$-dimensional minus-type summand let $s$ have characteristic polynomial $(x^2-\tau)^m$.  On a hyperbolic plane with basis $e,f$, choose $a\in\F_q^\times$ and set
\[
 s(e)=ae,\qquad s(f)=\tau a^{-1}f.
\]
Then $s$ has conformal multiplier $\tau$ on both summands.  Its determinant is
\[
 (-\tau)^m\tau=\tau^{m+1}=\tau^n,
\]
so $s\in CO^-_{2n}(q)^0$.  the nonsquare-multiplier description \cite[Lemma~2.5(i)]{Ng} gives
\[
 C_{O^-_{2n}(q)}(s)\cong O^-_{n-1}(q^2)\times \GL_1(q),
\]
up to the scalar conformal factor after passing to $CO^-_{2n}(q)^0$.

Let $s^*$ denote the image of $s$ in $G^*$.  The multiplier of $s$ is nonsquare, so $s^*\notin(G^*)'$.  Since $|Z(X)|=2$, \cite[Lemma~2.2]{HMII} implies that every character in $\mathcal E(X,s^*)$ lies over the nontrivial character of $Z(X)$.  Choose the semisimple character $\chi$ in this series.  Then $\chi$ is faithful.

Set $D=\chi(1)$ and suppose, for a contradiction, that $D/2=\psi(1)$ for some $\psi\in\Irr(S)$.  Inflate $\psi$ to $X$, and let $((t^*),\eta)$ be its Jordan decomposition.  Since the inflated character is trivial on $Z(X)$, \cite[Lemma 2.2]{HMII} gives $t^*\in(G^*)'$, so a lift $t$ has square conformal multiplier.  Comparing the two Jordan degree formulas gives
\begin{equation}\label{eq:Dminusoddonecent}
 |C_{G^*}(t^*)|_{p'}=2\eta(1)|C_{G^*}(s^*)|_{p'}.
\end{equation}
Because $D/2$ is a $p'$-number, $p\nmid\eta(1)$; Lemma~\ref{lem:Deveneta} yields $\eta(1)\in\{1,2,4\}$.

Write $n=2r+1$.  The $p'$-part of the $O^-_{n-1}(q^2)$ factor contains
\[
 q^{n-1}+1\quad\text{and}\quad q^{2n-6}-1.
\]
Choose primitive prime divisors $\rho$ and $\sigma$ with
\[
 \operatorname{ord}_{\rho}(q)=2n-2,\qquad
 \operatorname{ord}_{\sigma}(q)=2n-6.
\]
Zsigmondy's theorem supplies them for $n\ge7$ and odd $q$.  They are odd, so \eqref{eq:Dminusoddonecent} forces both into the exact square-multiplier centralizer of $t$: neither the projective component group nor $2\eta(1)$ can supply them.

Apply the square-multiplier decomposition \cite[Lemma~2.5(ii)]{Ng}.  A linear extension-field factor $\GL_a(q^k)$ cannot contain $\rho$, since this would require $2n-2\mid ki$ for some $i\le a$, whereas $ki\le ka\le n<2n-2$.

Suppose a unitary factor $\GU_a(q^k)$ contains $\rho$.  It must occur through a plus factor $q^{ki}+1$, and the rank bound forces
\[
 ki=n-1,\qquad ka\ge n-1.
\]
The same factor must contain $\sigma$, because at most one unit of rank remains.  A minus factor cannot supply $\sigma$, since it would require $2n-6\mid kj$ with $kj\le n$.  Hence $\sigma$ also comes from a plus factor, so
\[
 kj=n-3.
\]
Thus $k\mid\gcd(n-1,n-3)=2$.

If $k=1$, then $a\ge n-1$.  Hence the same unitary factor contains the odd-index factor $q^{n-2}+1$.  Let $\lambda_0$ be a primitive prime divisor of $q^{2n-4}-1$; then $\lambda_0\mid q^{n-2}+1$, so $\lambda_0$ divides $C_{G^*}(t^*)$.  But $\lambda_0$ does not divide $C_{G^*}(s^*)$: the nontrivial cyclotomic exponents in the latter come from $q^{n-1}+1$, the factors $q^{4i}-1$ with $1\le i\le r-1$, and the rank-one factor $q-1$, none of which has a primitive divisor of order $2n-4$.  This contradicts \eqref{eq:Dminusoddonecent}.

If $k=2$, then $i=(n-1)/2=r$.  Since $\rho$ comes from a plus unitary factor, $r$ must be odd.  The rank bound gives $a=r$.  But $j=(n-3)/2=r-1$ is then even, and the $j$th factor of $\GU_r(q^2)$ is $q^{2j}-1$, not $q^{2j}+1$.  It therefore cannot contain the primitive divisor $\sigma$ of order $2n-6=2(n-3)$.  This is again impossible.

Thus $\rho$ must lie in an orthogonal factor.  A proper orthogonal factor containing a primitive divisor of order $2n-2$ has rank at least $n-1$.  Rank $n$ would make the entire exact centralizer the full orthogonal group, so $t^*$ would be central and the corresponding character would have degree $1$, contrary to $D/2>1$.  Hence the only possibility has an $O^-_{2n-2}(q)$ factor and a residual rank-one factor.  But $O^-_{2n-2}(q)$ contains $q^{2n-4}-1$, and therefore contains the same primitive divisor $\lambda_0$ of order $2n-4$.  As above, $\lambda_0\nmid|C_{G^*}(s^*)|$, contradicting \eqref{eq:Dminusoddonecent}.

All square-multiplier possibilities are excluded.  Hence $D/2\notin\operatorname{cd}(S)$.
\end{proof}

\begin{theorem}\label{thm:Dminusoddonecodegree}
Let $n\ge7$ be odd, let $q\equiv1\pmod4$, and set $S=\POmega^-_{2n}(q)$.  If a nontrivial finite group $G$ satisfies
\[
 \operatorname{cod}(G)\subseteq\operatorname{cod}(S),
\]
then $G\cong S$.
\end{theorem}

\begin{proof}
Apply Proposition~\ref{prop:assemblyreduction}, using Proposition~\ref{prop:DminusoddoneB} for the elementary-abelian hypothesis and Theorem~\ref{thm:Dminusoddonecover} for the prime-central-cover hypothesis.
\end{proof}

\section{Plus-type even-dimensional orthogonal groups of even rank with $q\equiv1\pmod4$}\label{sec:Dplusevenone}
We now complete the large-rank plus-type family.  Throughout this section
\[
 S=\POmega^+_{2n}(q),\qquad n\ge8\text{ even},\qquad q=p^f\equiv1\pmod4.
\]
As in Subsection~\ref{subsec:Deven}, the generic multiplier of $S$ has order $4$ and
$Z(\Spin^+_{2n}(q))\cong C_2\times C_2$.

\begin{proposition}\label{prop:DplusevenoneB}
Conjecture B holds for $S=\POmega^+_{2n}(q)$ when $n\ge8$ is even and $q\equiv1\pmod4$.
\end{proposition}

\begin{proof}
This is the corresponding specialization of Proposition~\ref{prop:DallB}, proved uniformly for both signs, every odd $q$, and every $n\ge5$.  We state it separately here so that the prime-cover argument below can be read without referring back to the uniform formulation.
\end{proof}

\begin{theorem}\label{thm:Dplusevenonecover}
Let $n\ge8$ be even and $q\equiv1\pmod4$, and put
\[
 S=\POmega^+_{2n}(q),\qquad \widetilde S=\Spin^+_{2n}(q).
\]
For every subgroup $Z_0<Z(\widetilde S)$ of order $2$, the double cover
\[
 X=\widetilde S/Z_0
\]
has a faithful character $\chi\in\Irr(X)$ such that $\chi(1)/2\notin\operatorname{cd}(S)$.
Hence Theorem~D of \cite{HMII} holds for $S$.
\end{theorem}

\begin{proof}
Let
\[
 G^*=P(CO^+_{2n}(q)^0).
\]
As in the proof of Theorem~\ref{thm:Devencover}, the three nontrivial central characters consist of one nontrivial square-multiplier coset and two nonsquare-multiplier cosets, the latter two being exchanged by the type $D_n$ graph automorphism.  We treat one coset of each kind.

\smallskip
\noindent\emph{The nontrivial square-multiplier coset.}
Since $q\equiv1\pmod4$, choose a nondegenerate minus-type $2$-space $W$ and let $s$ act as $-I$ on $W$ and as $I$ on $W^\perp$.  As in the proof of Theorem~\ref{thm:Dfivecover}, the spinor norm of $s$ is nonsquare, so $s^*$ lies in the nontrivial square-multiplier coset.  The case $m=n$, $k=1$ in \cite[Lemma~2.5(ii)]{Ng} gives a faithful character of the associated double cover of degree
\[
 d_0=\frac{(q^n-1)(q^{n-1}-1)}{2(q+1)}.
\]
Put $d=d_0/2$.  Since $q\ge5$ and $n\ge8$, $d<q^{4n-10}$, so the complete list \cite[Theorem~1.4]{Ng} applies.  The $\eta=-1$, $\varepsilon=+1$ comparison in the proof of Theorem~\ref{thm:Dfivecover}, with $x=q^{n-1}$, shows that $d$ is different from every degree in that list.  Hence $d\notin\operatorname{cd}(S)$.

\smallskip
\noindent\emph{A nonsquare-multiplier coset.}
Write $n=2r$.  Choose a nonsquare multiplier $\tau$ and take the $m=n$ plus-type case in the nonsquare-multiplier centralizer description \cite[Lemma~2.5(i)]{Ng}.  Thus a lift $s\in CO^+_{2n}(q)^0$ may be chosen so that
\[
 C_{O^+_{2n}(q)}(s)\cong O^+_n(q^2).
\]
By Lemma~\ref{lem:extfieldcoset}, the nontrivial field automorphism of the
plus-type $O_n^+(q^2)$ factor is represented by an $\F_q$-isometry of
determinant $1$.  In the Case~1c construction in the proof of
\cite[Lemma~2.5(i)]{Ng}, the two relevant eigenvalues are $\lambda_1$ and
$-\lambda_1$, with $\lambda_1^2=\tau$ and $\lambda_1\notin\F_q$; hence
$\lambda_1^q=-\lambda_1$.  Thus the $q$-Frobenius interchanges the two
eigenspaces and sends $s$ to $-s$.  It therefore belongs to
$CO^+_{2n}(q)^0$ and gives the nontrivial projective-centralizer component.  Conversely, projective centralization has the form $gsg^{-1}=\lambda s$ with $\lambda^2=1$, so this component has order exactly $2$.

Since
\[
 |O^+_n(q^2)|_{p'}=2(q^n-1)\prod_{i=1}^{r-1}(q^{4i}-1),
\]
the associated semisimple character has degree
\begin{equation}\label{eq:DplusevenoneD}
 D=\frac12\prod_{\substack{1\le j<n\\ j\text{ odd}}}(q^{2j}-1).
\end{equation}
Its parameter has nonsquare multiplier, hence represents one of the two nonsquare-multiplier central characters.  Therefore the resulting character is faithful on the corresponding double cover; the graph automorphism will give the other one.

Suppose that $D/2=\psi(1)$ for some $\psi\in\Irr(S)$.  Inflate $\psi$ to $\widetilde S$, and let $((t^*),\eta)$ be its Jordan decomposition.  Then $t^*\in(G^*)'$ by \cite[Lemma 2.2]{HMII}, so a lift $t$ has square multiplier.  Since $D/2$ is a $p'$-number, Lemma~\ref{lem:Deveneta} gives $\eta(1)\in\{1,2,4\}$.  Comparing Jordan degrees yields
\begin{equation}\label{eq:Dplusevenonecent}
 |C_{G^*}(t^*)|_{p'}
 =4\eta(1)(q^n-1)\prod_{i=1}^{r-1}(q^{4i}-1).
\end{equation}

Let $\rho$ be a primitive prime divisor of $q^{2n-4}-1$.  It is odd and exists by Zsigmondy's theorem.  Equation~\eqref{eq:Dplusevenonecent} forces $\rho$ into the exact square-multiplier centralizer of $t$.  Use the decomposition \cite[Lemma~2.5(ii)]{Ng}.

If an orthogonal factor contains $\rho$, its rank $k$ is at least $n-2$.
Since an orthogonal factor of rank $k$ contains $q^{2i}-1$ for
$1\le i\le k-1$, it contains $q^{2n-6}-1$.  Let $\lambda$ be a primitive
prime divisor of this factor, so
$\operatorname{ord}_{\lambda}(q)=2n-6=2(n-3)$.  We check explicitly that
$\lambda$ cannot divide the right-hand side of
\eqref{eq:Dplusevenonecent}.  It cannot divide $q^n-1$, since
$2(n-3)\nmid n$ for even $n\ge8$.  If it divided $q^{4i}-1$ for some
$1\le i\le r-1=n/2-1$, then $2(n-3)\mid4i$.  As $n-3$ is odd, this would
imply $n-3\mid i$, impossible because $i\le n/2-1<n-3$.  Finally
$\lambda$ is odd, so it cannot come from the factor $4\eta(1)$, where
$\eta(1)\in\{1,2,4\}$.  This contradiction excludes every orthogonal
factor.

A linear extension-field factor cannot contain $\rho$, because $2n-4>n$ while every exponent $ki$ occurring in $\GL_a(q^k)$ satisfies $ki\le ka\le n$.  Hence $\rho$ lies in a unitary factor $\GU_a(q^k)$.  Its cyclotomic factors are $q^{ki}-(-1)^i$.  Since $\operatorname{ord}_{\rho}(q)=2(n-2)$ and $ki\le n$, the minus case $q^{ki}-1$ is impossible; thus $i$ is odd and
\[
 ki=n-2.
\]
In particular $ka\ge n-2$.  Let $\sigma$ be a primitive prime divisor of $q^n-1$.  The residual rank outside this unitary factor is at most $2$.  For $n\ge8$ it cannot contain $\sigma$: a linear or orthogonal residual factor has all relevant exponents at most $4<n$, and a unitary residual factor has $ki\le2$ (hence order at most $4$).  Therefore $\sigma$ also lies in $\GU_a(q^k)$.  If it occurs in a factor $q^{kj}-(-1)^j$, there are two possibilities.  If $j$ is even, then $q^{kj}-1$ contains $\sigma$, so $n\mid kj$ and hence $k\mid\gcd(n-2,n)=2$.  If $j$ is odd, then $q^{kj}+1$ contains $\sigma$; because $kj\le ka\le n$ and $kj<n$ in the plus case, primitivity forces $2kj=n$, so $k\mid\gcd(n-2,n/2)$.  The relation $ki=n-2$ with $i$ odd shows in particular that $k$ is even.  Now
\[
 \gcd(n-2,n/2)=
 \begin{cases}2,&4\mid n,\\1,&n\equiv2\pmod4,
 \end{cases}
\]
so the plus case also forces $k=2$ (and in fact $4\mid n$).  Thus in every case $k=2$, and $a\ge(n-2)/2=r-1$.  Since $2a\le n$, there are exactly two rank possibilities, $a=r-1$ and $a=r$.

There are then two rank possibilities.  Suppose first that $ka=n-2$, so $a=r-1$.  The even-index alternative for the primitive divisor $\sigma$ of $q^n-1$ is impossible, since $n\mid 2j$ would force $r\mid j$, whereas $j\le r-1$.  Hence $\sigma$ occurs in an odd-index factor $q^{2j}+1$.  Thus $4j=n=2r$, so $r$ is even.  The same centralizer must also contain a primitive divisor of order $2n-8$ coming from the factor $q^{2n-8}-1$ in $C_{G^*}(s^*)$.  In $\GU_{r-1}(q^2)$ the required index is $r-2$, which is now even, so the corresponding factor is $q^{2(r-2)}-1$, not $q^{2(r-2)}+1$.  The residual rank $2$ cannot supply a primitive divisor of order $2n-8\ge8$.  This is impossible.

It remains that $ka=n$, so the full exact centralizer has unitary type
$\GU_r(q^2)$.  Here $r$ is necessarily even.  Indeed, the primitive divisor
$\sigma$ of $q^n-1$ lies in a factor $q^{2j}-(-1)^j$ with $j\le r$.  If
$j$ is even, then $n\mid2j$, so $j=r$ and hence $r$ is even.  If $j$ is odd,
then primitivity gives $4j=n=2r$, again making $r$ even.  This centralizer is
too small.  At the determinant-one level the ratio of the relevant $p'$-parts is
\[
 \frac{|O^+_{2r}(q^2)|_{p'}}{|\GU_r(q^2)|_{p'}}
 =2\prod_{j=1}^{r-1}\bigl(q^{2j}+(-1)^j\bigr)>4.
\]
For both exact conformal centralizers, \cite[Lemma~2.5]{Ng} and \eqref{eq:projectiveimageorder} identify the exact projective $p'$-part with the corresponding determinant-one centralizer $p'$-part.  Thus the ratio displayed above is unchanged after passing to the exact projective images.
Lemma~\ref{lem:projcent2} can enlarge either projective image by at most
another factor $2$; for the chosen parameter $s$ that nontrivial component
already occurs.  Consequently even the largest possible projective
$\GU_r(q^2)$ centralizer is strictly smaller than
$C_{G^*}(s^*)$.  On the other hand, comparing \eqref{eq:Dplusevenonecent}
with the already computed projective centralizer of $s$ gives the exact
requirement
\[
 |C_{G^*}(t^*)|_{p'}=2\eta(1)\,|C_{G^*}(s^*)|_{p'}.
\]
Thus the hypothetical centralizer is at least twice, not ``at least four
times,'' the centralizer of $s$.  The strict inequality in the opposite
direction is already enough for the contradiction.  This excludes the
final possibility.

Thus $D/2\notin\operatorname{cd}(S)$.  We have treated one nonsquare-multiplier central character, and the graph automorphism treats the other.  Together with the square-multiplier case this handles all three double covers.
\end{proof}

\begin{theorem}\label{thm:Dplusevenonecodegree}
Let $n\ge8$ be even, let $q\equiv1\pmod4$, and put $S=\POmega^+_{2n}(q)$.  If $G$ is a nontrivial finite group satisfying
\[
 \operatorname{cod}(G)\subseteq\operatorname{cod}(S),
\]
then $G\cong S$.
\end{theorem}

\begin{proof}
Apply Proposition~\ref{prop:assemblyreduction}, using Proposition~\ref{prop:DplusevenoneB} for the elementary-abelian hypothesis and Theorem~\ref{thm:Dplusevenonecover} for every prime-central-cover hypothesis.
\end{proof}

\begin{corollary}\label{cor:Dlargeall}
Let $q$ be odd.  The codegree isomorphism conjecture holds for $\POmega^+_{2n}(q)$ when $n\ge7$ is odd or $n\ge8$ is even, and for $\POmega^-_{2n}(q)$ in the same large-rank ranges.
\end{corollary}

\begin{proof}
Combine Sections~\ref{subsec:Deven}, \ref{sec:Dplusone}, \ref{sec:Dminusoddone}, and \ref{sec:Dplusevenone} with the corresponding $q\equiv3\pmod4$ results above.
\end{proof}

\begin{theorem}\label{thm:Doddallcodegree}
Let $q$ be odd, let $n\ge4$, let $\varepsilon\in\{+,-\}$, and suppose that
\[
 S=\POmega_{2n}^{\varepsilon}(q)
\]
is simple.  If $G$ is a nontrivial finite group and
\[
 \cod(G)\subseteq\cod(S),
\]
then $G\cong S$.
\end{theorem}

\begin{proof}
For $4\le n\le6$ use Theorem~\ref{thm:Dlowcodegree}.  For the large-rank cases combine Theorems~\ref{thm:Doddcodegree}, \ref{thm:DevenCodegree}, \ref{thm:Dplusonecodegree}, \ref{thm:Dminuscodegree}, \ref{thm:Dminusevencodegree}, \ref{thm:Dminusoddonecodegree}, and \ref{thm:Dplusevenonecodegree}, according to the sign, parity of $n$, and the residue of $q$ modulo $4$.
\end{proof}

\section{Projective symplectic groups: the residual elementary-abelian case}
We now turn to the complementary reduction for projective symplectic groups.  Let
\[
 H=\PSp_{2n}(q),\qquad q=p^f.
\]
By \cite[Theorem~5.1(1)--(3)]{TV}, a residual minimal counterexample to the codegree isomorphism conjecture has $n\ge4$ and a unique minimal normal subgroup $N$ such that $G/N\cong H$; the subgroup $N$ is an elementary abelian $p$-group, every irreducible $\overline{\F}_pL$-constituent of $N\otimes_{\F_p}\overline{\F}_p$, where $L=\Sp_{2n}(q)$, is quasi-equivalent to the natural $L$-module, and
\[
 |N|=q^{2n}.
\]  He also proves that every faithful $\chi\in\Irr(G)$ in such a minimal counterexample satisfies
\[
 |N|\mid\chi(1)
\]
\cite[Lemma 2.3(ii)]{TV}.

We first make explicit the field-of-definition step needed below.

\begin{lemma}\label{lem:sympdescent}
Assume that $q=p^f$ and that $N$ is the irreducible $\F_pH$-module occurring
in the residual conclusion of \cite[Theorem~5.1]{TV}.  Then $N$ is the restriction of scalars
of a $2n$-dimensional $\F_qH$-module which, after twisting the fixed
identification $G/N\cong H$ by an automorphism of $H$, is the natural
symplectic module whenever $q$ is even.
\end{lemma}
\begin{proof}
Put $E=\operatorname{End}_{\F_pH}(N)$.  By Schur's lemma and Wedderburn's theorem,
$E$ is a finite field, say $E=\F_{p^e}$.  The action of $E$ commutes with $H$, so $N$ is naturally an $EH$-module.  Since $\operatorname{End}_{EH}(N)=E$, this $EH$-module is absolutely irreducible.  Under the $e$ embeddings of $E$ into $\overline{\F}_p$, scalar extension decomposes as the direct sum of the $e$ distinct Frobenius conjugates of one absolutely irreducible constituent $W$ of $N\otimes_{\F_p}\overline{\F}_p$.  Consequently
\[
 \dim_{\F_p}N=e\dim_{\overline{\F}_p}W.
\]
By \cite[Theorem~5.1(2)]{TV}, $W$ is quasi-equivalent to the natural
module, so $\dim W=2n$.  By \cite[Theorem~5.1(3)]{TV},
$|N|=q^{2n}=p^{2nf}$, whence $\dim_{\F_p}N=2nf$.  Therefore $e=f$ and
$E=\F_q$.  Thus the original module $N$, not merely a chosen scalar
extension constituent, is a $2n$-dimensional $\F_qH$-module.  When $q$ is
even, $H=\Sp_{2n}(q)$; the quasi-equivalence in \cite[Theorem~5.1]{TV}
allows the fixed isomorphism $G/N\cong H$ to be precomposed with the
corresponding automorphism, after which this $\F_qH$-module is the natural
module.
\end{proof}

We next isolate the projective-character argument.

\begin{lemma}\label{lem:projind}
Let $S$ be finite, let $L\le S$, and let $\alpha\in H^2(S,\C^\times)$. If $\operatorname{res}^S_L(\alpha)=0$, then $S$ has an irreducible $\alpha$-projective character $\tau$ satisfying
\[
 \tau(1)\le [S:L].
\]
\end{lemma}
\begin{proof}
Choose a cocycle representing $\alpha$ whose restriction to $L\times L$ is trivial. Projectively induce the trivial one-dimensional representation of $L$ to $S$. The resulting $\alpha$-projective representation has degree $[S:L]$, so any irreducible constituent has degree at most $[S:L]$.
\end{proof}

\begin{proposition}\label{prop:sympeven}
Let $q=2^f$, $n\ge4$, and let
\[
 1\longrightarrow N\longrightarrow G\longrightarrow \Sp_{2n}(q)\longrightarrow1
\]
be an extension in which $N$ is the natural $2n$-dimensional $\F_q$-module viewed additively. Then there exists $\chi\in\Irr(G\mid N)$ such that
\[
 \chi(1)_2<|N|=q^{2n}.
\]
If $N$ is the unique minimal normal subgroup of $G$, then $\chi$ is faithful.
\end{proposition}
\begin{proof}
Fix $1_N\ne\lambda\in\Irr(N)$.  The symplectic form identifies the dual module with the natural module, and $\Sp_{2n}(q)$ is transitive on its nonzero vectors.  If $T=I_G(\lambda)$ and $S=T/N$, then
\[
 [G:T]=q^{2n}-1.
\]
The stabilizer of a nonzero vector has the form
\[
 S=R\rtimes L,\qquad L\cong\Sp_{2n-2}(q),\qquad |R|=q^{2n-1}.
\]
The character triple $(T,N,\lambda)$ determines a factor set $\alpha\in H^2(S,\C^\times)$.

Suppose first that $(n,q)\ne(4,2)$.  Here $L\cong\Sp_{2n-2}(q)$ has trivial Schur multiplier.  Indeed, for $n\ge5$ this follows from the generic multiplier formula, and for $n=4$ with $q>2$ the same conclusion holds for $\Sp_6(q)$; the only exception in this range is $\Sp_6(2)$.  See \cite[Theorem~5.1.4 and Tables~5.1.A, 5.1.D]{KL}.  Hence $H^2(L,\C^\times)=1$, so $\alpha|_L$ is trivial.  Lemma~\ref{lem:projind} gives an irreducible $\alpha$-projective character $\tau$ of $S$ with
\[
 \tau(1)\le [S:L]=q^{2n-1}.
\]
By character-triple correspondence there is $\varphi\in\Irr(T\mid\lambda)$ with $\varphi(1)=\tau(1)$, and Clifford correspondence gives
\[
 \chi=\varphi^G\in\Irr(G),\qquad
 \chi(1)=(q^{2n}-1)\tau(1).
\]
Since $q^{2n}-1$ is odd,
\[
 \chi(1)_2=\tau(1)_2\le\tau(1)\le q^{2n-1}<q^{2n}=|N|.
\]

It remains to treat $(n,q)=(4,2)$.  Here $S\cong2^7{:}\Sp_6(2)$ and $[G:T]=255$.  If the extension $N{:}\Sp_8(2)$ is split, then $\lambda$ extends to $T$ by taking it to be trivial on the vector stabilizer, and Clifford induction gives an irreducible character of degree $255$.

Assume that the extension is nonsplit.  Dempwolff \cite[p.~451, Corollary]{Dem} computes the relevant second cohomology; in the present case,
\[
 H^2(\Sp_8(2),V)\cong\F_2
\]
for the natural $8$-dimensional module $V$.  Thus there is exactly one nonzero extension class for this fixed natural action; uniqueness of the nonsplit extension is also stated in \cite[p.~451, Theorem]{Dem}.  Basheer and Moori \cite{BM} study precisely the nonsplit extension
$G_4=2^8\!\cdot\Sp_8(2)$ with this natural action; hence $G_4$ represents the unique nonzero class.  Therefore the nonsplit extension arising here is equivalent, and hence isomorphic, to $G_4$.

No interpretation of a projective factor set is needed.  In the Clifford--Fischer decomposition of \cite[Section~9.5]{BasheerThesis}, the $195$ irreducible characters of $G_4$ are partitioned into an $81$-character part corresponding to $H_1=\Sp_8(2)$ and a $114$-character part corresponding to the nonprincipal orbit on $\Irr(N)$, whose inertia factor is
\[
 H_2=2^7{:}\Sp_6(2).
\]  In the complete character table, Table~11.13, the first character in this second Clifford--Fischer part is $\chi_{82}$ and
\[
 \chi_{82}(1)=255
\]
\cite[Table~11.13, p.~274]{BasheerThesis}.  Consequently $G_4$ has an irreducible character of degree $255$ lying over the nonprincipal orbit on $\Irr(N)$, so it is nontrivial on $N$.  Hence
\[
 \chi(1)_2=1<2^8=|N|.
\]

Finally suppose that $N$ is the unique minimal normal subgroup of $G$.  Since $\chi$ lies nontrivially over $N$, $\Ker(\chi)\cap N$ is a proper $G$-invariant subgroup of the irreducible module $N$, hence is trivial.  If $\Ker(\chi)\ne1$, uniqueness of the minimal normal subgroup forces $N\le\Ker(\chi)$, a contradiction.  Therefore $\chi$ is faithful.
\end{proof}

\begin{remark}\label{rem:sympTV}
For the minimal-counterexample argument of \cite[Theorem~5.1]{TV}, Proposition~\ref{prop:sympeven} is stronger than necessary: the displayed degree already shows $|N|\nmid\chi(1)$, contradicting \cite[Lemma 2.3(ii)]{TV}, with the identification supplied by Lemma~\ref{lem:sympdescent}.
\end{remark}

\section{Projective symplectic groups in odd characteristic}\label{sec:sympodd}
The odd-characteristic residual case is excluded by comparing the kernels of the residual module and the natural module.

\begin{proposition}\label{prop:sympodd}
Let $q$ be odd and $n\ge4$.  Then the residual minimal counterexample described in \cite[Theorem~5.1]{TV} does not exist for $H=\PSp_{2n}(q)$.
\end{proposition}

\begin{proof}
Let $L=\Sp_{2n}(q)$ and $H=L/Z(L)$.  By \cite[Theorem~5.1(1)--(3)]{TV}, if $W$ is an irreducible constituent of $N\otimes_{\F_p}\overline{\F}_p$, then $W$, regarded as an $L$-module through $L\to H$, is quasi-equivalent to the natural $L$-module.  Quasi-equivalent representations differ only by an automorphism of $L$ and an equivalence of the underlying modules; in particular, faithfulness is preserved.  The natural $L$-module is faithful.  On the other hand, since the action on $W$ factors through $H=L/Z(L)$ and $q$ is odd, the nontrivial center $Z(L)=\langle-I\rangle$ lies in the kernel of the action on $W$.  Thus $W$ is not faithful as an $L$-module, a contradiction.
\end{proof}

\begin{theorem}\label{thm:sympall}
Let $H=\PSp_{2n}(q)$ be nonabelian simple.  If $G$ is a finite group satisfying
\[
 \cod(G)=\cod(H),
\]
then $G\cong H$.
\end{theorem}

\begin{proof}
If $n=1$, then $\PSp_2(q)\cong\PSL_2(q)$, so the result follows from \cite{BAK}.  Theorem~5.1 of \cite{TV} proves that a minimal counterexample in the remaining symplectic family must have $n\ge4$; in particular Theorem~5.1 treats the cases $n=2,3$.  Let $n\ge4$ and suppose that a minimal counterexample exists.  By \cite[Theorem~5.1(1)--(3)]{TV}, $G$ has a unique elementary abelian defining-characteristic minimal normal subgroup $N$, its scalar-extension constituents are quasi-equivalent to the natural module, and
$|N|=q^{2n}$.  If $q$ is odd, Proposition~\ref{prop:sympodd} gives a contradiction.  If $q$ is even, Lemma~\ref{lem:sympdescent} identifies the residual module, after twisting the quotient identification if necessary, with the natural $\Sp_{2n}(q)$-module.  Proposition~\ref{prop:sympeven} then produces a faithful $\chi\in\Irr(G)$ with $\chi(1)_2<|N|$; in particular $|N|\nmid\chi(1)$, contrary to \cite[Lemma 2.3(ii)]{TV}.  Thus no minimal counterexample exists.
\end{proof}

\section{Completion of the proof}\label{sec:completion}

\begin{proposition}\label{prop:assemblyreduction}
Let $H$ be a nonabelian finite simple group.  Assume that Conjecture~B holds for $H$.  If
\[
 H\not\cong \POmega_{2n}^{\pm}(q)\qquad(n\ge4,\ q\text{ odd}),
\]
then Conjecture~A holds for $H$ by \cite[Theorem~3.1]{HMII}.  If
$H\cong \POmega_{2n}^{\pm}(q)$ with $n\ge4$ and $q$ odd, assume in addition that for every quasisimple group $X$ with
\[
 Z(X)\cong C_r\quad(r\text{ prime}),\qquad X/Z(X)\cong H,
\]
there is a faithful $\chi\in\Irr(X)$ such that
\[
 \chi(1)/r\notin\operatorname{cd}(H).
\]
Then Conjecture~A holds for $H$ also in this exceptional orthogonal case.  Thus in either case every nontrivial finite group $G$ satisfying
$\operatorname{cod}(G)\subseteq\operatorname{cod}(H)$ is isomorphic to $H$.
\end{proposition}

\begin{proof}
The containment statement here is exactly Conjecture~A of \cite{HMII}, not merely the equality version of the codegree conjecture; see \cite[Conjecture~A]{HMII}.  We use the proof of \cite[Theorem~3.1]{HMII} in its stated minimal-counterexample form.  That proof first shows that a minimal counterexample $G$ is perfect, has a unique minimal normal subgroup $N$, and satisfies $G/N\cong H$.  It then proves that $N$ is elementary abelian of order a power of a prime dividing $|H|$ and splits into exactly two cases: $C_G(N)=G$ or $C_G(N)=N$.  In the first case $N=Z(G)$ and $G$ is a quasisimple group whose center has prime order; the contradiction is supplied there by Theorem~D.  In the second case $N$ is a faithful irreducible $H$-module and Conjecture~B supplies the contradiction.  No orthogonal-family hypothesis is used before Theorem~D is applied in the first case.

Most importantly, this consequence is stated explicitly in \cite{HMII} immediately after the proof of Theorem~3.1: if Theorem~D could be proved for even-dimensional orthogonal groups, then the same argument would show that Conjecture~B implies Conjecture~A for those groups as well \cite[the note immediately following Theorem~3.1]{HMII}.  Thus no extrapolation of their reduction is needed here.  For odd-characteristic $\POmega_{2n}^{\pm}(q)$, the additional hypothesis of the present proposition is exactly the conclusion of Theorem~D for every possible prime central quasisimple cover.  It supplies the sole missing result in the central case, while Conjecture~B handles the self-centralizing case exactly as in that proof.  Hence Conjecture~A follows.
\end{proof}
\begin{proof}[Proof of Theorem~\ref{thm:global}]
We now combine the results proved above with the known cases from the literature.  By the classification of finite simple groups, it is enough to verify each family of nonabelian finite simple groups.  The sporadic and alternating cases are \cite{DMSSYsporadic,DMSSYalternating}.  The order-divisibility and simple-quotient reductions used throughout are \cite[Theorem~C and Theorem~8.3]{HMI}.  The rank-one linear groups $\PSL_2(q)$ are covered by \cite{BAK}, and $\PSL_3(q)$ by \cite{LY}.  The simple exceptional groups of Lie type covered there and $\PSL_n(q)$ for $n\ge4$ are covered by \cite[Theorem~1.1]{TV}.  The Tits group ${}^2F_4(2)'$ is covered by \cite{WZZC}.  The exceptional low-parameter derived groups are already covered by the isomorphisms $G_2(2)'\cong \PSU_3(3)$ and ${}^2G_2(3)'\cong \PSL_2(8)$; see \cite[Remark~24.18]{MT}.

It remains to consider the other classical families.  We use the main results of this paper as follows.  Theorem~\ref{thm:Umain}, followed by Corollary~\ref{cor:Umain}, proves the required conclusion for the unitary groups $\PSU_n(q)$ for $n\ge4$; the rank-three groups $\PSU_3(q)$ are covered by \cite{LY}.  For the symplectic groups, \cite[Theorem~5.1]{TV} treats $\PSp_4(q)$ and $\PSp_6(q)$, and Theorem~\ref{thm:sympall} treats $\PSp_{2n}(q)$ for $n\ge4$.  For odd-dimensional orthogonal groups in odd characteristic, Theorem~\ref{thm:Bmain} proves Conjecture B and Corollary~\ref{cor:Bmain} gives the required codegree conclusion for every $n\ge3$.  The lower odd-characteristic ranks satisfy $\POmega_3(q)\cong \PSL_2(q)$ and $\POmega_5(q)\cong \PSp_4(q)$, while in characteristic $2$ the odd-dimensional orthogonal groups coincide with the corresponding symplectic groups.  For even-dimensional orthogonal groups in characteristic $2$, Theorem~\ref{thm:mainB} proves Conjecture B and Corollary~\ref{cor:mainA} gives the codegree conclusion for both signs and all $n\ge4$.  The omitted simple lower ranks are already covered through $\POmega_4^-(q)\cong \PSL_2(q^2)$, $\POmega_6^+(q)\cong \PSL_4(q)$, and $\POmega_6^-(q)\cong \PSU_4(q)$.  Finally, in even dimension and odd characteristic, the prime-cover and elementary-abelian arguments developed in Sections~\ref{sec:DallB}--\ref{sec:Dplusevenone} give Corollary~\ref{cor:Dlargeall} and Theorem~\ref{thm:Doddallcodegree}; the latter treats both signs for every rank $n\ge4$ in the simple range, including rank $4$.  These results cover all of the classical families not already supplied by the cited literature.  The standard low-rank coincidences produce overlaps, not omissions.  Together with the alternating, sporadic and exceptional groups cited above, this exhausts the CFSG list of finite nonabelian simple groups.  Hence $G\cong H$.
\end{proof}

\section{Classical groups and the $p$-group reduction}
Hypothesis~5.4 of \cite{HMII} asks for a degree bound for central extensions of Sylow $p$-subgroups of a classical simple group, and Proposition~5.5 there shows that this implies Conjecture B \cite[Hypothesis 5.4 and Proposition 5.5]{HMII}.  The arguments above use substantially more global structure.  In the orthogonal and unitary families the small-module reduction leaves only the natural module; in the symplectic residual case the residual symplectic conclusion \cite[Theorem~5.1]{TV} reaches the same natural-size configuration.  A single vector stabilizer then replaces an arbitrary Sylow extension.  Thus the proof proceeds by
\[
\begin{aligned}
 \text{global module reduction}&\longrightarrow \text{vector stabilizer}\\
 &\longrightarrow \text{multiplier of a classical subgroup}\\
 &\longrightarrow \text{Clifford or projective induction}.
\end{aligned}
\]
The $p$-group hypothesis remains a useful uniform sufficient condition, but it is stronger than what is required in these orthogonal and unitary configurations and in the symplectic residual reduction used here.

\section*{Acknowledgements}
This work was partially supported by a grant from the Simons Foundation (\#918096, to YY).

\section*{Disclosure Statement}
The author declares that there are no competing interests.

\section*{Data Availability Statement}
Data sharing is not applicable to this article, as no data sets were generated or analysed during the current study.

\end{document}